\documentclass[11pt]{article}
\usepackage{amsmath, latexsym, amsfonts, amssymb, amsthm, amscd}
\usepackage{xcolor}
\usepackage[hidelinks]{hyperref}
\usepackage{graphicx}
\usepackage{placeins}
\usepackage{subfigure}
\usepackage{caption}
\DeclareGraphicsExtensions{.pdf,.png,.jpg,.eps}
\hypersetup{
breaklinks,
colorlinks=true,
citecolor=blue
}

\newtheorem{theorem}{Theorem}

\newtheorem{proposition}{Proposition}
\newtheorem{lemma}{Lemma}
\newtheorem{rem}{Remark}

\graphicspath{{img/}}  

\title{
\textbf{Ray-Knight Theorems for Spectrally Negative L\'evy Processes}}

\author{\textbf{ Jes\'us Contreras\footnote{Centro de Investigaci\'on en Matem\'aticas A.C. Calle Jalisco s/n. C.P. 36240, {\sc Guanajuato, Mexico.}  Email: jjcontreras@cimat.mx.}\,\,,  V{\'i}ctor Rivero\footnote{Centro de Investigaci\'on en Matem\'aticas A.C. Calle Jalisco s/n. C.P. 36240, {\sc Guanajuato, Mexico.}  Email: rivero@cimat.mx.}\, .}}
\date{\footnotesize This version: \today}

\begin{document}
\maketitle

\begin{abstract}
\bigskip

In this paper, we study the law of the local time processes $(L_T^x(X),x\in \mathbb{R})$ associated to a spectrally negative L\'evy process $X$, in the cases $T=\tau_a^+$, the first passage time of $X$ above $a>0$ and $T=\tau(c)$, the first time it accumulates $c$ units of local time at zero. We describe the branching structure of local times and Poissonian constructions of them using excursion theory. The presence of jumps for $X$ creates a type of excursions which can contribute simultaneously to local times of levels above and below a given reference point. This fact introduces dependency on local times, causing them to be non-Markovian. Nonetheless, the overshoots and undershoots of excursions will be useful to analyze this dependency. In both cases, local times are infinitely divisible and we give a description of the corresponding L\'evy measures in terms of excursion measures. These are hence analogues in the spectrally negative L\'evy case of the first and second Ray-Knight theorems, originally stated for the Brownian motion.
\bigskip

\noindent {\sc Key words}:  Local times, L\'evy processes, Fluctuation theory, Ray-Knight theorem, Infinite divisibility \\
\noindent MSC 2010 subject classifications: 60J99, 60G51.
\end{abstract}

\vspace{0.5cm}

\section{Introduction} \label{intro4}
Let $X=(X_t,t\geq 0)$ be a real valued spectrally negative L\'evy process, that is, a stochastic process with independent and stationary increments and no positive jumps. Denote by $S=(S_t,t\geq 0)$ its running supremum. Its Laplace transform exists, characterizes its law and can be expressed as
$$\mathbb{E}\left[e^{\lambda X_t} \right]=e^{t\Psi(\lambda)},\quad t,\lambda \geq 0,$$ where the function $\Psi$ is called the Laplace exponent of $X$. $\Psi$ can be expressed by the L\'evy-Khintchine formula $$\Psi(\lambda)=d\lambda+\frac{1}{2}\Sigma^2\lambda^2+\int_{(-\infty,0)}(e^{\lambda x}-1-\lambda x1_{\{-1<x\}})\Pi(dx),$$ where the triplet $(d,\Sigma,\Pi)$ consists of $d,\Sigma \in \mathbb{R}$ and a $\sigma-$finite measure $\Pi$ over $(-\infty,0)$ satisfying the condition $\int_{(-\infty,0)} (1\wedge x^2) \Pi(dx) <\infty$. 

From now on, we consider $X$ to be a SNLP such that:
\begin{itemize}
\item[\textbf{(A)}] it is of unbounded variation (which is equivalent to $\Sigma\neq 0$ or $\int_{-1}^0 |x| \Pi(dx)=\infty$),
\end{itemize}
and either
\begin{itemize}
\item[\textbf{(B1)}] $X_t \to \infty$ a.s. as $t\to \infty$ (which is equivalent to the condition $\Psi'(0+)>0$);
\end{itemize}
 or
\begin{itemize}
\item[\textbf{(B2)}] $X$ oscillates as $t\to \infty$ (which is equivalent to the condition $\Psi'(0+)=0$).
\end{itemize}
One can actually remove these conditions, but we will not tackle this task here. For deeper insight in the theory of L\'evy processes, we cite \cite{bertoin1996levy}, \cite{kyprianou2014fluctuations} and \cite{ken1999levy} as our standard references.

We are interested in studying the local time process associated to $X$, which is here denoted by $\left(L_t^x(X),t\geq 0, x\in \mathbb{R} \right)$. Essentially, one interprets $L_t^x(X)$ as the amount of time $X$ stays at level $x$ on the interval $[0,t]$. Formally, $L_t^x(X)$ is defined as the a.s. limit
\begin{equation}\label{ltdef}
L_t^x(X) := \limsup_{\varepsilon \to 0^+} \frac1{2\varepsilon} \int_0^t 1_{\{|X_s-x|<\varepsilon \}}ds, \ \ \ x\in \mathbb{R}, t\geq 0.
\end{equation}
Under rather general conditions (see \cite[Ch.V]{bertoin1996levy}), the convergence holds also in $L^2$ and uniformly over compact sets of $t$. Furthermore, local times satisfy the so-called occupation density formula, that is, $$\int_0^T f(X_s)ds\stackrel{a.s.}{=}\int_\mathbb{R} f(x)L_T^x (X) dx,$$ for any $f\geq 0$ measurable and bounded and stopping times $T$.

Our aim is to describe the local time process $(L_T^x(X),x\in \mathbb{R})$, where $T$ is a fixed stopping time. In general, local times indexed by the spatial variable are not easy to describe. Nonetheless, this has been a matter of research interest since early times of the theory of stochastic processes, perhaps originated by the pioneering works of Ray and Knight in the decade of 1960. In particular, the theory around the so-called isomorphism theorems, has proven to be one powerful tool to study properties of local times. 

Ray and Knight completely characterized the law of the local times in the case $X$ is a standard Brownian motion and $T$ is either the first time $X$ is above a positive level $a$ or the first time it accumulates a certain amount of local time at zero. These results are known as the first and second Ray-Knight theorems, respectively, and they are expressed in terms of squared Bessel processes as follows.

\begin{theorem}[First Ray-Knight theorem]\label{RK1}
Let $X$ be a Brownian motion issued from zero, $a>0$ a fixed level and $$\tau_a^+=\inf\{t>0: X_t> a\}$$ the first passage time above $a$. Then,
\begin{itemize}
\item[i)] the process $(L^{a-z}_{\tau_a^+}(X), z \in [0,a])$ has the same law as a squared Bessel process of dimension 2 started from 0;
\item[ii)] the process $(L^{a-z}_{\tau_a^+}(X), z\geq a)$ has the same law as a squared Bessel process of dimension 0, issued from $L^{0}_{\tau_a^+}(X).$ 
\end{itemize}
Moreover, conditionally on $L^0_{\tau_a^+}(X)$ the two parts are independent.
\end{theorem}

\begin{theorem}[Second Ray-Knight theorem]\label{RK2}
Let $X$ be a Brownian motion issued from zero, $c>0$ a constant and $$\tau(c)=\inf\{t>0:L_t^0(X) >c \}$$ the first time $0$ has accumulated $c$ units of local time. Then, the process $(L^y_{\tau(c)}(X),y\geq 0)$ is distributed as a squared Bessel process of dimension $0$, started from $c$. By symmetry of $X$, the law of the process $(L^{-y}_{\tau(c)}(X),y\geq 0)$ is also that of a squared Bessel process of dimension $0$ started from $c$ and it is independent from the first one.
\end{theorem}

Further details on these results can be found for instance in \cite{marcus2006markov} and \cite{mansuy2008aspects}. Because of technical reasons and convenience on the narrative, we will present first the results concerning an analogue of the second Ray-Knight theorem and then the first. For Brownian motion, both theorems can be expressed using excursion theory (we refer to Section~\ref{preliminaries} for notation and more information on excursions). Indeed, if $N_0$ and $\overline{N}$ are the measures of the excursions away from 0 for $X$ and the reflected process $S-X$ on the space $D(0,\infty)$ of c\`adl\`ag paths, respectively, then local times up to $\tau(c)$ have the representation
\begin{equation}\label{PRbrow2}
L_{\tau(c)}^y (X)=\int_0^c \int_{D(0,\infty)} \ell (y) \widetilde{K}(ds,d\ell ), \quad y\in \mathbb{R}
\end{equation}
where $\widetilde{K}$ is a Poisson random measure related to $N_0$. The local time process up to $\tau_a^+$ can be expressed as
\begin{equation}\label{PRbrow}
L_{\tau_a^+}^{a-z}(X)=\int_0^{z\wedge a} \int_{D(0,\infty)}\ell (z-s)K(ds,d\ell), \quad z \geq 0,
\end{equation}
where $K$ is also a Poisson random measure but related to $\overline{N}$. Actually, it is known that $N_0=2\overline{N}$, which can be seen as a consequence of L\'evy's identity $(|X_t|,L_t^0(X))_{t\geq 0} \stackrel{(d)}{=} (S_t-X_t,S_t)$ for Brownian motion. Since these measures are multiples of each other, equations (\ref{PRbrow2}) and (\ref{PRbrow}) can be written in terms of a single Poisson random measure. See Theorems \ref{PRtauC} and \ref{PRtauA} below for a proof of these representations in a more general setting.

\begin{figure}[htb]
\centering
\subfigure[]{\includegraphics[width = 0.45\textwidth]{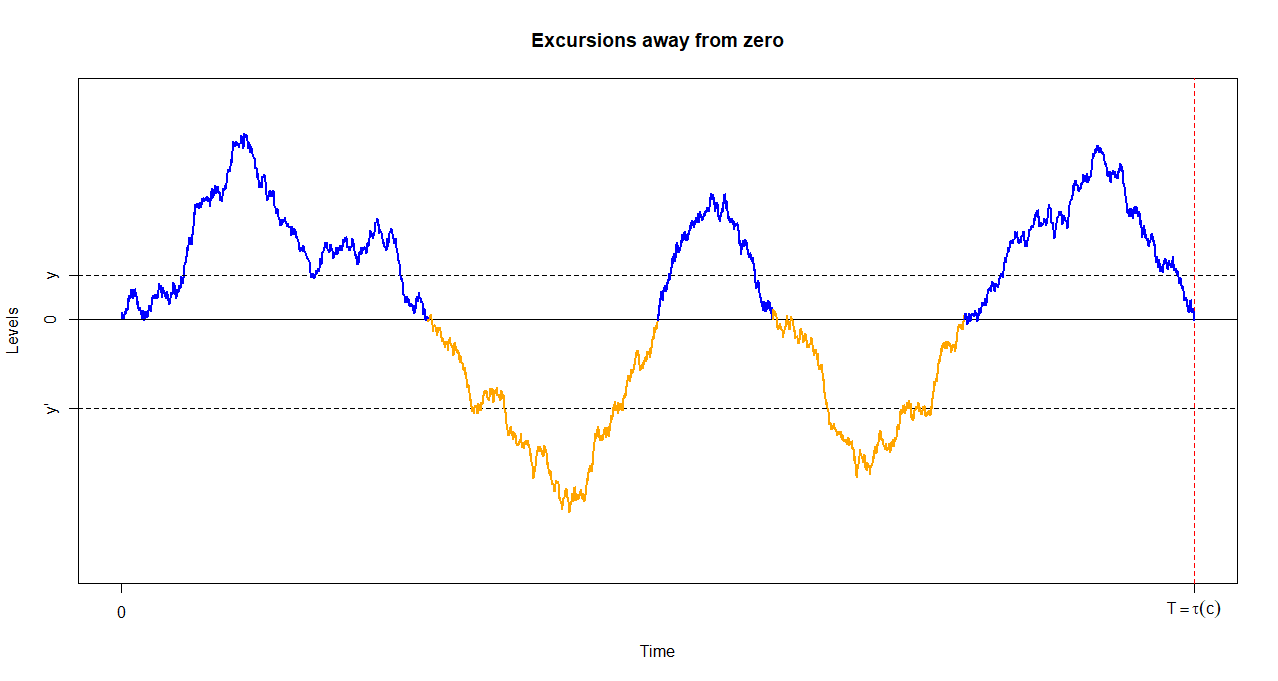}}
\subfigure[]{\includegraphics[width = 0.45\textwidth]{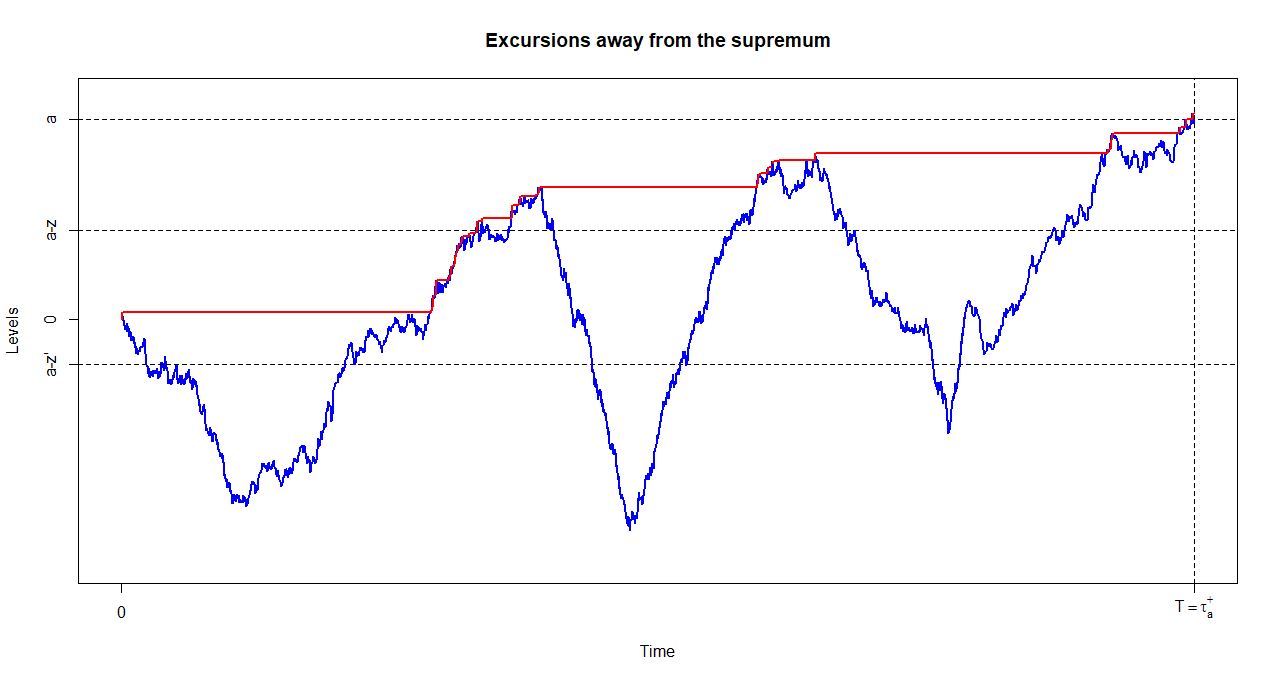}} 
\caption{\small \it (a) Graphic representation of equation \eqref{PRbrow2}. Excursions away from 0 happen until 0 accumulates $c$ units of local time and each one can contribute to the local time of levels $y>0$ and $y'<0$. (b) Graphic representation of equation \eqref{PRbrow}. For the level $a-z$ with $z\in [0,a]$, its local time comes from the contribution of each excursion away from the supremum starting above that level, whilst for $a-z'$ with $z'\geq a$ all excursions can contribute.}
\label{rkplots}
\end{figure}
\FloatBarrier

On both Ray-Knight theorems, the properties of Brownian excursions play a key role to prove the independence of the two parts involved. Consider for instance the second Ray-Knight theorem, in which local times of levels above and below zero are independent. Since the Brownian motion has continuous paths, one can split the set of excursions away from zero into two disjoint sets: the excursions completely above zero, $\mathcal{E}_+$, and the ones completely below, $\mathcal{E}_-$ (see also Figure~\ref{rkplots}(a)). Observe that only excursions in $\mathcal{E}_+$ contribute to the local time of positive levels, and the same is true for $\mathcal{E}_-$ and local times of negative levels. So, given the fact that excursions form a Poisson point process, its restrictions to $\mathcal{E}_+$ and $\mathcal{E}_-$ are independent and this translates to the processes $(L_{\tau(c)}^y(X),y\geq 0)$ and $(L_{\tau(c)}^{-y}(X),y\geq 0)$. The independence of the restrictions of the process of excursions away from a point to disjoint sets can also be used to explain the Markov property of local times. Indeed, conditionally on the information of the local time at a reference level, local times of levels above and below are independent. 

Another important remark about both theorems is that a squared Bessel process of dimension 2 started from zero is in particular a continuous state branching process (CSBP) with linear immigration $\psi(z)=2z$ and branching mechanism $\phi(z)=z^2/2$. A squared Bessel process of dimension $0$ started from $c$ is also a CSBP with the same branching mechanism $\phi$ but no immigration. They satisfy the stochastic differential equations $$Z_t=\int_0^t \sqrt{Z_s}d\beta_s + 2t, \quad t\geq 0,$$ and $$Z_t=c+\int_0^t \sqrt{Z_s}d\beta_s, \quad t\geq 0,$$ respectively, where $\beta$ is a Brownian motion. These facts reflect the branching nature of local times and bring interest on finding out if similar stochastic integral equations are satisfied by them in a more general setting.

We would like to explore if we can obtain analogues of the Ray-Knight theorems in the more general case in which the process does not have continuous trajectories, more specifically for SNLPs. One important issue here is a consequence of a result in \cite{eisenbaum1993necessary}. N. Eisenbaum and H. Kaspi proved that if the local times of the process $X$ have the Markov property, then $X$ necessarily has continuous paths. This implies that for instance, if $X$ is a Lévy process with only negative jumps, the processes $(L_{\tau_a^+}^{a-z}(X),z\geq 0)$ and $(L_{\tau(c)}^{y}(X),y \in \mathbb{R})$ are not Markovian. Hence, trying to characterize the law of the local time process yields to a finer study of the structure of dependence between the local times at different levels. We will explore this via excursions. Additional to the sets $\mathcal{E}_+$ and $\mathcal{E}_-$, an excursion away from a point, say 0, can also be an element of a third set: $\mathcal{E}_\pm$ (see also the forthcoming Figure~\ref{plots}). This set consists of excursions starting above 0 and then jumping below, hence having the possibility to contribute to local times of both positive and negative levels. The positions $(\mathcal{O}_0,\mathcal{U}_0)$ of the excursion prior to and at the first passage time below 0 are called the overshoot and the undershoot, respectively, and will be relevant to describe the law of local times.

Nonetheless, recent efforts have been made in this direction for Lévy processes. In \cite{li2018fluctuations}, B.Li and Z. Palmowski gave an expression for the Laplace transform of functionals of the form $\int_0^T dt f(X_t)$ in terms of generalized scale functions, where $T=\tau_a^+ \wedge \tau_b^-$ is the first exit time from the interval $[b,a]$. These functionals are strongly related to local times via the occupation density formula mentioned before. Later, in \cite{li2020local} B. Li and X. Zhou obtained joint Laplace transforms also in terms of generalized scale functions and related the law of the local times under a change of measure with permanental processes. 

In 2023, J. Contreras and V. Rivero \cite{contreras2022generalized} extended Li and Palmowski's results for functionals also involving the supremum $S$ of $X$. Actually, we can make use of the results there in order to gain some information on the local time process up to $\tau_a^+$, $(L_{\tau_a^+}^{a-z}(X),z\geq 0)$. According to their notation, for a measurable and bounded function $f:\mathbb{R} \to \mathbb{R}_+$, the generalized scale function $W_f$ is defined by 
\begin{equation} \label{scalef}
W_f(x,b)=W(x-b)\exp\left\{\int_b^x ds \overline{N}\left(1-e^{-\int_0^\zeta dr f(s-\mathbf{e}(r))} ,H<s-b\right) \right\}, \quad x \geq b>-\infty,
\end{equation}
which, up to a positive constant, is also a solution of the integral equation $$W_f(u,v)=W(u-v) + \int_v^u dz W(u-z)f(z)W_f(z,v).$$ In case the function $f$ is a constant, say $f\equiv q$, $W_f$ is the usual $q-$scale function associated to $X$ and when $f\equiv 0$ we will just write $W$ (see \cite{kuznetsov2012theory} for further reference). In this paper, the function $\mathcal{W}_f : [0,\infty) \to \mathbb{R}_+$ defined by
\begin{equation} \label{newScaleF}
\mathcal{W}_f(x):= \lim_{b\to -\infty} \frac{W_f(0,b)}{W_f(x,b)} = \exp\left\{ -\int_0^x ds \overline{N} \left(1-e^{-\int_0^\zeta dr f(s-\mathbf{e}(r))} \right)\right\}, \quad x \geq 0,
\end{equation}
will play an important role. See Sections~\ref{srkt} and \ref{frkt} to observe its connections with the analogues of the Ray-Knight theorems for the spectrally negative case.

Finally, in \cite{xu2021ray}, W. Xu explored the local times of a spectrally positive $\alpha-$stable process $Y$ up to $\tau(c)$. Since $X:=-Y$ is spectrally negative, a consequence of their results is that conditioned on $\tau(c)<\infty$, the process $(L_{\tau(c)}^{-y}(X),x\geq 0)$ has the law of a non-Markovian branching system which they called \textit{rough continuous state branching process}. This class of processes is characterized for being weak solutions to certain stochastic Volterra equations, and in particular $(L_{\tau(c)}^{-y}(X),y\geq 0)$ satisfies $$Z_y=c(1-bW(y))+\int_0^y \int_0^\infty \int_0^{Z_s} (W(y-s)-W(y-s-u))\widetilde{N}_\alpha (ds,du,dz),$$ where $b$ is the drift of $Y$, $W$ is its scale function and $\widetilde{N}_\alpha$ is certain compensated Poisson random measure. Note that this stochastic equation is more involved than those presented before for Bessel processes. We do not intend to tackle this point of view in this article and let it for future studies.

The content of the article is organized as follows. Section~\ref{preliminaries} is dedicated to introduce notation and recall some properties of excursions. In Sections~\ref{srkt} and \ref{frkt} we state our main results. In Section~\ref{srkt}, we perform a study of local times constituting an analogue of the second Ray-Knight theorem. Here, we provide a Poissonian construction for $(L_{\tau(c)}^{y}(X),y \in \mathbb{R})$, proving that it is infinitely divisible and giving its corresponding L\'evy measure and emphasizing the importance of the overshoots and undershoots with respect to zero. Regarding the process $(L_{\tau_a^+}^{a-z}(X),z \geq 0)$, Section~\ref{frkt} contains an analogue of the first Ray-Knight theorem. We derive a Poissonian construction and describe its law as an infinitely divisible process, focusing on the joint law of local times under $\overline{N}$. Section~\ref{decomp} is based on the decomposition of the L\'evy measure of an infinitely divisible process found in \cite{eisenbaum2019decompositions}, applying the ideas there to the L\'evy measures of the processes of local times. Finally, Section~\ref{auxR} contains some results which are useful for the main theorems and Section~\ref{prfs} compiles all the proofs.

\section{Preliminaries on excursions}\label{preliminaries}

First, we deal with the excursions away from the supremum, or equivalently, the excursions away from 0 for $S-X$. Let us introduce the space $\mathcal{E}$ of positive right continuous paths with left limits and defined on an interval: $$\mathcal{E}=\left\{\mathbf{e}:[0,\zeta] \to [0,\infty)  \ | \ \zeta \in (0,\infty], \ \mathbf{e}((0,\zeta))\subset (0,\infty) \text{ and }  \mathbf{e} \text{ is c\`adl\`ag}  \right\}.$$ This space is usually regarded as a subset of the more general space $D(0,\infty)$ of c\`adl\`ag paths, so we might write one or the other depending on the context. For an element $\mathbf{e}\in \mathcal{E}$, the right endpoint of its interval of definition is called duration or lifetime and it is denoted by $\zeta(\mathbf{e})$. The supremum of $\mathbf{e}$ is called height and it is denoted by $H(\mathbf{e})=\sup_{v\in [0,\zeta]} \mathbf{e}(v)$.

For a SNLP $X$, it is well known that one can take $S$ as the local time at the supremum and that its right continuous inverse is the subordinator $(\tau_t^+,t\geq 0)$. For $t>0,$ such that $\tau_t^+ \neq \tau_{t-}^+ := \lim_{s\uparrow t}\tau_s^+$, the supremum is constant and equal to $t$ on the interval $[\tau_{t-}^+,\tau_t^+]$. Therefore, we can define $$\mathbf{e}_t(v):=(S-X)_{\tau_{t-}^+ +v}, \ \ \ 0\leq v\leq \tau_t^+ - \tau_{t-}^+,$$ the excursion of $(S-X)$ at local time $t$. In this case, $\mathbf{e}_t \in \mathcal{E}$ and actually, $\zeta(\mathbf{e}_t)=\tau_t^+ - \tau_{t-}^+$. In case $\tau_t^+=\tau_{t-}^+$, one assigns $\mathbf{e}_t=\delta$, where $\delta \notin \mathcal{E}$ is an auxiliary state. See also \cite{blumenthal2012excursions} and \cite{fitzsimmons2006excursion} for more information.

A result of excursion theory (see for example \cite[Th. 6.14]{kyprianou2014fluctuations}) states that there exists a measure space $(\mathcal{E},\Sigma,\overline{N})$ such that $\Sigma$ contains the sets of the form $$\left\{\mathbf{e}\in \mathcal{E} : \zeta(\mathbf{e})\in A, H(\mathbf{e})\in B, \mathbf{e}(\zeta)\in C \right\},$$ where $A,B,C$ are Borel sets on $\mathbb{R}$. Furthermore, if $\limsup_{t \to \infty}X_t=\infty$ a.s. (that is, iff $\Psi'(0+)\geq 0$), then $\{(t,\mathbf{e}_t) : t>0, \mathbf{e}_t \neq \delta \}$ is a Poisson point process of intensity $ds\otimes \overline{N} (d\mathbf{e})$. This fact explains our hypotheses \textbf{(B1)} and \textbf{(B2)} in Section~\ref{intro4}. In the other case, if $\Psi'(0+)<0$ one obtains a killed Poisson point process and it requires a different treatment. For a deeper insight on this excursion measure, we refer to \cite{chaumont2005levy} and \cite{duquesne2003path}.

For excursions away from a point we have a similar situation, but we need a different subordinator. Taking as a reference a given point $y\in \mathbb{R}$, there exists an associated subordinator $\sigma^y =(\sigma^y_s,s\geq 0)$ defined by the right inverse of the local time as follows $$\sigma_s^y=\inf\{t>0: L_t^y(X) >s \}, \quad s\geq 0.$$ To formally define the excursions away from a point, we work with $\sigma^y$. For each $s>0$, $\sigma_s^y$ corresponds to the first time the process $X$ accumulates $s$ units of local time at level $y$. We can define the excursions by looking at the constancy times of $L_\cdot ^y(X)$, or equivalently, at the increasing times of $\sigma^y$. Indeed, for each $u\geq 0$ such that $\sigma^y_u > \sigma^y_{u-}:=\lim_{s\uparrow u} \sigma^y_s$, $$\mathbf{e}_u^y(t)=X_{\sigma^y_{u-}+t}, \qquad 0\leq t\leq \sigma^y_{u}-\sigma^y_{u-}$$ and now the quantity $\zeta(\mathbf{e}_u^y):=\sigma^y_{u}-\sigma^y_{u-}$ is called the length of the excursion. The excursion process is a Poisson point process on the space $[0,\infty)\times \mathcal{E}^y$, where $\mathcal{E}^y$ is the space of c\`adl\`ag paths with lifetime, starting and ending at $y$, and the intensity is given by the product of the Lebesgue measure and the so-called It\^o measure: $ds \otimes N_y(d\mathbf{e})$. Again, this space can be regarded as a subspace of $D(0,\infty)$. More information on the excursion measure away from a point can be found in \cite{pardo2018excursion}. The previous paragraphs allow to use the tools from the theory of Poisson point processes, such as the compensation and exponential formulas, to perform computations related to excursions of $X$.

We end this section with the following Proposition involving scale functions and quantities related to $\overline{N}$, which will be useful in some of the results later. 

\begin{proposition}\label{intBarN}
Let $f:\mathbb{R} \to \mathbb{R}_+$ a measurable and bounded function. Then, the function $\mathcal{W}_f$ in \eqref{newScaleF} is well defined and it is a solution of the following equation involving the scale function $W$ of $X$:
\begin{equation} \label{intBarN1}
F(x)=F(0)-\int_0^\infty dz(W(x)-W(x-z)) f(x-z) F(z), \quad x > 0.
\end{equation} 
Define $$G_f (x)= \widehat{\mathbb{E}}_x \left[\exp\left\{-\int_0^{\tau_0^-} ds f(x-X_s) \right\} \right], \quad  x\geq 0,$$ and $$g_f(x)=\overline{N} \left[1-e^{-\int_0^\zeta dr f(x-\mathbf{e}(r))}  \right]=\overline{N} \left[1-e^{-\int_0^\infty dy f(x-y) L_\zeta^y }  \right], \quad x\in \mathbb{R}.$$ Then, $$G_f(x)=\mathcal{W}_f(x)=\exp\left\{ -\int_0^x ds g_f(s)\right\},$$ and in particular we have the relation $$\frac{d}{dx}(-\log G_f)(x) = g_f(x).$$
\end{proposition}

\section{Second Ray-Knight theorem}\label{srkt}

Recall that a non-negative infinitely divisible process $\psi=(\psi_x,x\in E)$ is characterized by its L\'evy measure $\mu$, and as such it satisfies 
\begin{equation*}
\mathbb{E}\left[e^{-\int_\mathbb{E} f(x)\psi_x dx} \right]=\exp\left\{-\int_{\mathbb{R}_+ ^\mathbb{R}} \left(1-e^{-\int_\mathbb{E} f(x)\omega(x)dx} \right) \mu(d\omega) \right\},
\end{equation*}
for any non-negative, measurable and bounded function $f$. In \cite{eisenbaum2019decompositions}, N. Eisenbaum provides a deeper understanding of the L\'evy measure by decomposing it into two parts, essentially corresponding to the information of a process between the first and last visits to a point and the complement. In our case, we will explore this decomposition in Section~\ref{decomp}.

We begin with the following theorem, which proves the local times up to $\tau(c)$ are infinitely divisible and also provides a Poissonian representation of them.

\begin{theorem} \label{PRtauC}
Let $X$ be a spectrally negative L\'evy process and denote by $N_0$ the associated excursion measure away from zero for $X$. Assume $X$ satisfies hypothesis {\normalfont\textbf{(A)}} and {\normalfont\textbf{(B2)}}. Then, the local time process $(L_{\tau(c)}^y(X),y\in \mathbb{R})$ is infinitely divisible and its L\'evy measure $\mu^{(c)}$ is given by $$\mu^{(c)}(d\omega)=cN_0(L_\zeta^\cdot \in d\omega),$$ that is, the image of the excursion measure away from zero $N_0$ under the function that maps an excursion into its local time process up to its lifetime. Moreover, local time process admits the representation
\begin{equation}
L_{\tau(c)}^y (X)=\int_0^c \int_{D(0,\infty)} \ell (y) \widetilde{K}(ds,d\ell ), \quad y\in \mathbb{R},
\end{equation}
where $\widetilde{K}$ is a Poisson random measure of intensity $ds\otimes \widetilde{M}(d\ell)$, $\widetilde{M}$ being the image of $N_0$ under the map that associates an excursion $\mathbf{e}$ its local time process up to its lifetime $\zeta$: $\mathbf{e}\mapsto \left(L_\zeta ^r (\mathbf{e}),r\in \mathbb{R} \right)$.
\end{theorem}

The representation for local times is similar to that in \cite[Ch.6]{li2020continuous} for CSBP processes. In that context, the intensity of the Poisson random measure $\widetilde{K}$ is $ds\otimes Q_H$ and $Q_H$ is a measure that has the information of both an entrance law and the transitions of a branching process and it is known as the Kuznetsov measure (see \cite[Ch.XIX]{dellacherie1992processus}). In our setting, $\widetilde{M}$ cannot be a Kuznetsov measure, since otherwise local times would be Markovian.

Actually, we can elaborate on these results a bit more, but we need to introduce formally the concept of overshoot and undershoot of a path. These quantities tell us the relative position of the path prior and at the first passage time below a given level $x$. In general, for a c\`adl\`ag path $Y$ having only negative jumps and a level $x$ such that $Y_0 \geq x$, we denote by $(\mathcal{O}_x(Y),\mathcal{U}_x(Y))$ to the pair of values $$(\mathcal{O}_x(Y),\mathcal{U}_x(Y))=(Y_{\tau_x^-(Y)-}-x, Y_{\tau_x^-(Y)}-x),$$ where $\tau_x^-(Y)=\inf\{t>0: Y_t<x\}$. Notice that if $Y$ crosses below $x$ continuously, both quantities are equal to zero. That is the reason we did not see their role in the Brownian motion case, since it has continuous paths. Nonetheless, since in this case the trajectories have negative jumps, they will arise naturally and will be really important for path decompositions.

According to \cite{pardo2018excursion}, $N_0$ is now carried by the partition of the space of excursions into the sets $\mathcal{E}_+ \sqcup \mathcal{E}_- \sqcup \mathcal{E}_\pm$, which consist in completely positive, completely negative and mixed excursions, respectively. Observe that an excursion $\mathbf{e} \in \mathcal{E}_+$ contributes only to the local times of positive levels. Similarly, $\mathbf{e} \in \mathcal{E}_-$ only contributes to negative levels. But, unlike the Brownian motion case, when the process has one-sided jumps there is an additional kind of excursions, $\mathcal{E}_\pm$, which can add to the local time of both positive and negative levels. As a consequence, we cannot split the information of local times above and below a point into functions of disjoint sets of excursions, hence losing the independence that for instance the Brownian motion had in Theorems~\ref{RK1} and \ref{RK2}. This also translates to the absence of the Markov property for the process of local times.

\begin{figure}[htb]
\centering
\subfigure[]{\includegraphics[width = 0.32\textwidth]{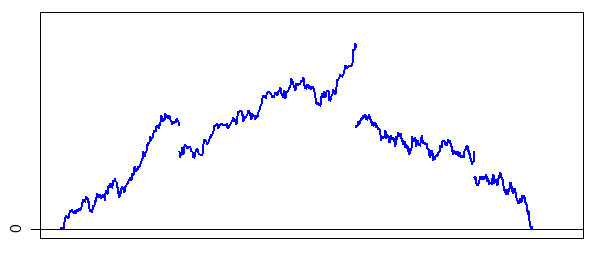}} 
\subfigure[]{\includegraphics[width = 0.32\textwidth]{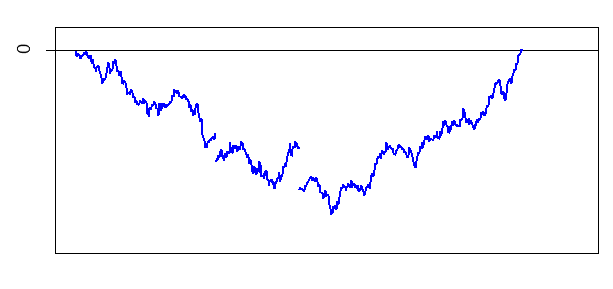}}
\subfigure[]{\includegraphics[width = 0.32\textwidth]{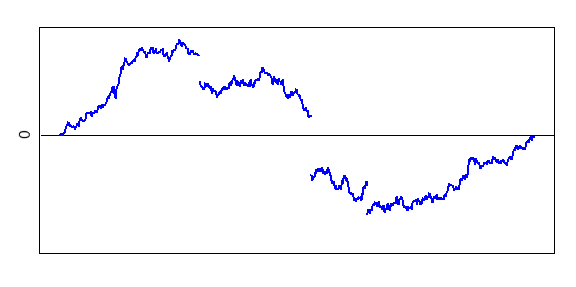}}
\caption{\small Typical excursions away from zero in (a) $\mathcal{E}_+$, (b) $\mathcal{E}_-$ and (c) $\mathcal{E}_\pm$.}
\label{plots}
\end{figure}
\FloatBarrier

A typical excursion $\mathbf{e} \in \mathcal{E}_\pm$ starts positive and then it jumps below 0, at time $\tau_0^-(\mathbf{e})$. After this time, since the path does not have positive jumps, it will creep upwards 0 and end at its lifetime $\zeta(\mathbf{e})$. Therefore, we can decompose $\mathbf{e}$ into the paths $\underleftarrow{\mathbf{e}}$ and $\underrightarrow{\mathbf{e}}$, where $$\underleftarrow{\mathbf{e}}(t)=\mathbf{e}((\tau_0^-(\mathbf{e})-t)-), \quad 0\leq t \leq \tau_0^-(\mathbf{e})$$ and $$\underrightarrow{\mathbf{e}}(t)=\mathbf{e}(\tau_0^-(\mathbf{e})+t), \quad 0\leq t \leq \zeta(\mathbf{e})-\tau_0^-(\mathbf{e}).$$ As we mentioned before, the position relative to 0 before and at first passage are called the overshoot and undershoot at 0 of the path, that is, $\mathcal{O}_0(\mathbf{e})=\mathbf{e}(\tau_0^--)>0$ and $\mathcal{U}_0(\mathbf{e})=\mathbf{e}(\tau_0^-)<0$. By the strong Markov property at time $\tau_0^-,$ it turns out that $\underleftarrow{\mathbf{e}}$ and $\underrightarrow{\mathbf{e}}$ are conditionally independent given $(\mathcal{O}_0,\mathcal{U}_0)$. Moreover, the reversed path $\underleftarrow{\mathbf{e}}$ has the same law as the dual $\widehat{X}$ started from $\mathcal{O}_0$ and killed at its first hitting time of $0$, denoted by $\widehat{\mathbb{E}}_{\mathcal{O}_0}^{0}$. On the other hand, $\underrightarrow{\mathbf{e}}$ has the same law as $X$ started from $\mathcal{U}_0$ and killed at its first hitting time of 0, denoted by $\mathbb{E}_{\mathcal{U}_0}^{0}$.
We can actually compute the joint law of $(\mathcal{O}_0,\mathcal{U}_0)$ in the set $\mathcal{E}_\pm$ under $N_0$, as can be read in Lemma~\ref{ouN0} (see Section~\ref{auxR}). This helps to obtain the following refined version of Theorem~\ref{PRtauC}.

\begin{theorem} \label{LevyMc}
Let $X$ be a spectrally negative L\'evy process satisfying {\normalfont\textbf{(A)}} and {\normalfont\textbf{(B2)}}. Then, the L\'evy measure $\mu^{(c)}$ associated to the local time process $(L_{\tau(c)}^y(X),y\in \mathbb{R})$ can be decomposed as $$\mu^{(c)}(d\omega)=cN_0(L_\zeta^\cdot \in d\omega)=cN_0(L_\zeta^\cdot \in d\omega,\mathcal{E}_+)+cN_0(L_\zeta^\cdot \in d\omega,\mathcal{E}_-)+cN_0(L_\zeta^\cdot \in d\omega,\mathcal{E}_\pm),$$ and when restricted to $\mathcal{E}_\pm$,
\begin{align*}
N_0(L_\zeta^\cdot \in d\omega,\mathcal{E}_\pm) &=\int_{(0,\infty)} \int_{(-\infty,0)} db \Pi(du-b) \widehat{\mathbb{E}}_{b}^{0} \otimes \mathbb{E}_{u}^{0}(L_\zeta^\cdot \in d\omega).
\end{align*}
In particular, for any $f:\mathbb{R}\to [0,\infty)$ measurable and bounded, if we write $f=f_++f_-$, with $f_+=f1_{(0,\infty)}$ and $f_-=f1_{(-\infty,0)}$, the Laplace transform of $\int_\mathbb{R} f(y)L_{\tau(c)}^y(X) dy$ is given by
\begin{align*}
\mathbb{E}\left[e^{-\int_\mathbb{R} f(y)L_{\tau(c)}^y(X) dy} \right]&= \exp\left\{-cN_0\left[1-e^{-\int_\mathbb{R} f(y)L_\zeta^y dy} \right] \right\} \\
&= \exp\left\{-cN_0\left[1-e^{-\int_{(0,\infty)} f_+(y)L_\zeta^y dy},\mathcal{E}_+ \right] \right\} \exp\left\{-cN_0\left[1-e^{-\int_{(-\infty,0)} f_-(y)L_\zeta^y dy},\mathcal{E}_- \right] \right\}\\
&\times \exp\left\{-c \int_{(0,\infty)} \int_{(-\infty,0)} db \Pi(du-b)\left[1- \mathcal{W}_{f_{+,b}}(b)\mathcal{W}_{f_{-,u}}(-u) \right] \right\},
\end{align*}
where $\mathcal{W}_{f_{+,b}}$ and $\mathcal{W}_{f_{-,u}}$ are defined as in \eqref{newScaleF} with $f_{+,b}(r):=f_+(-r+b)$ and $f_{-,u}(r):=f_-(r+u)$.
\end{theorem}

In particular, in Xu's $\alpha-$stable case \cite{xu2021ray}, there is no Brownian component, (and so $N_0$ is only carried by $\mathcal{E}_\pm$) and they only consider negative levels, the expression above simplifies to
\begin{eqnarray*}
\mathbb{E}\left[e^{-\int_{(-\infty,0)} f_-(y)L_{\tau(c)}^y(X) dy} \right]&=&\exp\left\{-c \int_{(-\infty,0)} N_0\left(\mathcal{U}_0 \in du,\mathcal{E}_\pm \right) \left[1-\mathbb{E}^0_u\left(e^{-\int_{(-\infty,0)} f_-(y)L_\zeta^y dy} \right) \right] \right\} \\
&=& \exp\left\{-c \int_{(0,\infty)} db \int_{(-\infty,0)} \Pi(du-b)  \left[1-\mathcal{W}_{f_{-,u}}(-u) \right]\right\}.
\end{eqnarray*}

The consideration on the overshoots and undershoots also allow to provide a further Poissonian construction involving them. Some results on transformations of Poisson point processes will be used and we refer to \cite[Ch. 4]{resnick1992adventures} and \cite[Ch. 12]{kallenberg1997foundations} for further details. 

Suppose that $X$ satisfies hypotheses \textbf{(A)} and \textbf{(B2)} and let us assume for now that it does not have a Brownian component. Then, all excursions of $X$ away from 0 belong to the set $\mathcal{E}_\pm$. Using the additive property of local times, for any $y>0$, $L_{\tau(c)}^y(X)$ can be decomposed as the sum of the local time at $y$ of each excursion $(s,\mathbf{e}^0_s)$ away from 0 up to local time $c$. Observe that, for a generic excursion $\mathbf{e}^0$, given its overshoot $b=\mathcal{O}_0(\mathbf{e}^0)$, the law of the reversed path $\underleftarrow{\mathbf{e}}^0$ under $N_0$ is $\widehat{\mathbb{E}}_{b}^0$, that is, the same law as $\widehat{X}$ started from $b$ and killed at the first hitting time of zero. Therefore, conditionally on $\mathcal{O}_0(\mathbf{e}^0)$, this path can be decomposed again in excursions away from the infimum, say $\{(r,\underline{\mathbf{e}}_r): 0<r<\mathcal{O}_0(\mathbf{e}^0) \}$, and express the local time at $y$ in terms of the local time at $y-r$ of $\underline{\mathbf{e}}_r$. A similar computation can be performed for $L_{\tau(c)}^z(X)$, for any $z<0$, by decomposing the path $\underrightarrow{\mathbf{e}}^0$ into excursions away from the supremum, say $\{(v,\overline{\mathbf{e}}_v): \mathcal{U}_0(\mathbf{e}^0) <v<0 \}$. Observe that the law $\underline{\widehat{N}}$ of the excursions away from the infimum for $\widehat{X}$ coincides with $\overline{N}$ because of duality and hence everything can be expressed in terms of this measure. The figure below shows a graphic representation of this decomposition. 

\begin{figure}[htb]
\centering
\includegraphics[width = 0.6\textwidth]{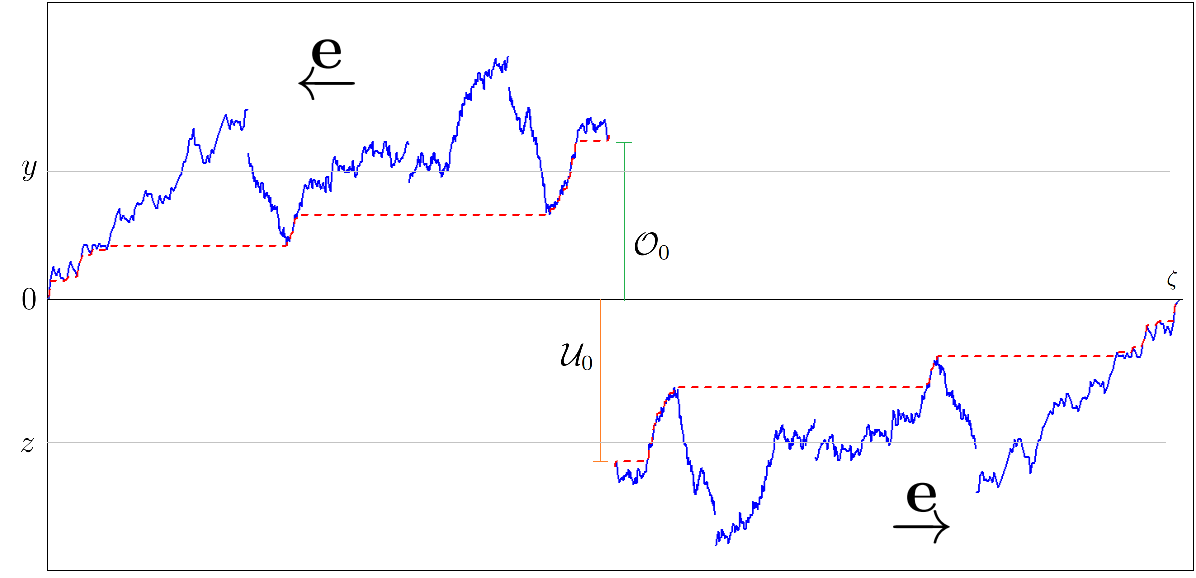}
\caption{\small \it Representation of a typical excursion away from zero $\mathbf{e} \in \mathcal{E}_\pm$ split into $\underleftarrow{\mathbf{e}}$ and $\underrightarrow{\mathbf{e}}$. The reversed path $\underleftarrow{\mathbf{e}}$ is further decomposed into excursions away from the infimum which contribute to the local time of levels $y>0$ and $\underrightarrow{\mathbf{e}}$ is decomposed into excursions away from the supremum, contributing to the local time of levels $z<0$.}
\end{figure}
\FloatBarrier

Repeating the above construction for each excursion away from 0, we have that the excursions away from zero can be seen as an atom in a marked Poisson point process that can be written as
\begin{equation} \label{atoms}
\{(s,\mathcal{O}_0(\mathbf{e}^0_s),\mathcal{U}_0(\mathbf{e}^0_s),((r,\underline{\mathbf{e}}^s_r), 0<r<\mathcal{O}_0(\mathbf{e}^0_s)),((v,\overline{\mathbf{e}}^s_v), \mathcal{U}_0(\mathbf{e}^0_s)<v<0)),s>0\},
\end{equation}

To be more formal, denote by $\mathcal{M}$ the space of Poisson random measures over $(0,\infty)\times D(0,\infty)$. We define a Markovian kernel $\widehat{\kappa}$ from $(0,\infty)$ to $\mathcal{M}$ by 
\begin{equation} \label{kaphat}
\widehat{\kappa}(b,dY)=\mathbb{P} \left(\sum_{0< r < b} \delta_{(r,\mathbf{e}_r)} \in dY \right), \quad b>0, Y\in \mathcal{M},
\end{equation} that is, the law of a random variable taking values in $\mathcal{M}$, which is given by a Poisson random measure with intensity $dr 1_{\{r\in (0,b) \}} \overline{N}(\mathbf{e} \in d\omega)$. Analogously, we define a kernel $\kappa$ from $(-\infty,0)$ to $\mathcal{M}$ by 
\begin{equation} \label{kap}
\kappa(u,dZ)=\mathbb{P} \left(\sum_{u< v <0} \delta_{(v,\mathbf{e}_v)} \in dZ \right), \quad u<0, Z \in \mathcal{M},
\end{equation}
which corresponds to the law of a random variable on $\mathcal{M}$, which is a Poisson random measure of intensity $dv 1_{\{v\in (u,0) \}}\overline{N}(\mathbf{e} \in d\omega)$.

Observe that, since the set $\{(s,\mathbf{e}_s^0):0<s\leq c \}$ of excursions away from 0 form a Poisson point process, then the corresponding overshoots and undershoots $$\{(s,(\mathcal{O}_0(\mathbf{e}_s^0),\mathcal{U}_0(\mathbf{e}_s^0))):0<s\leq c \}$$ are also a Poisson point process, now with intensity $\widetilde{m}_\pm(ds,db,du):= ds \otimes N_0(\mathcal{O}_0 \in db,\mathcal{U}_0 \in du,\mathcal{E}_\pm)$. Therefore, the intensity of the marked process is given by $$m_\pm(ds,db,du,dY,dZ)=\widetilde{m}_\pm(ds,db,du) \widehat{\kappa}(b,dY) \kappa(u,dZ).$$ Let $M_\pm(ds,db,du,dY,dZ)$ be a Poisson random measure on $(0,\infty)^2 \times (-\infty,0) \times \mathcal{M}^2$ with intensity $m_\pm(ds,db,du,dY,dZ)$. By the above considerations, we can describe the local time process in terms of $M_\pm$. Moreover, since the generic Poisson random measures $Y$ and $Z$ can be written as $Y=\sum_{0<r<b} \delta_{(r,\underline{\mathbf{e}}_r)}$ and $Z=\sum_{u<v<0} \delta_{(v,\overline{\mathbf{e}}_v)}$, with an abuse of notation we can regard $M_\pm$ as a Poisson random measure over $(0,\infty)^2 \times (-\infty,0)\times ((0,\infty) \times D(0,\infty))^2$ with intensity $m_\pm(ds,db,du,dr,d\underline{\mathbf{e}},dv,d\overline{\mathbf{e}})$, where $(dr,d\underline{\mathbf{e}})$ and $(dv,d\overline{\mathbf{e}})$ are the atoms in \eqref{atoms}. In order to provide separate expressions for local times of positive and negative levels, we denote by $$M^1_\pm(ds,db,dr,d\mathbf{e})=M_\pm(ds,db,(-\infty,0),dr,d\mathbf{e},(0,\infty),D(0,\infty)),$$ the restriction of $M_\pm$ to the information on the overshoots and by $$M^2_\pm(ds,du,dv,d\mathbf{e})=M_\pm(ds,(0,\infty),du,(0,\infty),D(0,\infty),dv,d\mathbf{e}),$$ the marginal related to the undershoots.

Now, if we remove the condition that $X$ does not have a brownian component, we additionally have to deal with excursions belonging to $\mathcal{E}_+$ and $\mathcal{E}_-$. Let $M_+(ds,d\mathbf{e})$ and $M_-(ds,d\mathbf{e})$ be Poisson random measures on $(0,\infty)\times D(0,\infty)$ with intensities $ds N_0(d\mathbf{e},\mathcal{E}_+)$ and $ds N_0(d\mathbf{e},\mathcal{E}_-)$, respectively. Since $\mathcal{E}_+$, $\mathcal{E}_-$ and $\mathcal{E}_\pm$ form a partition, the measures $M_+$, $M_-$ and $M_\pm$ are independent. Hence, we obtain the following Poissonian representation for local times of positive and negative levels up to $\tau(c)$.

\begin{theorem} \label{PRov}
Let $X$ be a SNLP satisfying {\normalfont\textbf{(A)}} and {\normalfont\textbf{(B2)}}. Then, the following Poissonian representations hold:
\begin{equation} \label{PRovPos}
L_{\tau(c)}^y(X)=\int_0^c \int_{D(0,\infty)} L_{\zeta(\mathbf{e})}^{y}(\mathbf{e}) M_+(ds,d\mathbf{e}) + \int_0^c \int_0^\infty \int_0^{b\wedge y} \int_{D(0,\infty)} L_{\zeta(\mathbf{e})}^{y-r} \left(\mathbf{e} \right) M^1_\pm(ds,db,dr,d\mathbf{e})
\end{equation}
for all $y>0$, and
\begin{equation} \label{PRovNeg}
L_{\tau(c)}^z(X)=\int_0^c \int_{D(0,\infty)} L_{\zeta(\mathbf{e})}^{z}(\mathbf{e}) M_-(ds,d\mathbf{e}) + \int_0^c \int_{-\infty}^0 \int_{u \vee z}^{0} \int_{D(0,\infty)} L_{\zeta(\mathbf{e})}^{v-z} \left(\mathbf{e} \right) M^2_\pm(ds,du,dv,d\mathbf{e}),
\end{equation}
for all $z<0$.
\end{theorem}

\begin{rem}
Recall that in the case of continuous paths, the processes of local times $(L_{\tau(c)}^y(X),y \geq 0)$ and $(L_{\tau(c)}^z(X),z \leq 0)$ are independent without any conditioning, as in Theorem~\ref{RK2}. In our case that is not true but as a consequence of the above Poissonian construction, if we condition to the whole Poisson point process of overshoots and undershoots at 0, the independence is recovered (notice that $M_+$ and $M_-$ do not alter the conditional independence, since they are carried by disjoint sets of excursions).
\end{rem}

\section{First Ray-Knight theorem}\label{frkt}

As in the previous section, we start showing that local times up to $\tau_a^+$ are infinitely divisible and admit a Poissonian representation.

\begin{theorem}\label{PRtauA}
Let $X$ be a spectrally negative L\'evy process and denote by $\overline{N}$ the associated excursion measure away from zero for $S-X$. Assume $X$ satisfies hypothesis {\normalfont\textbf{(A)}} and {\normalfont\textbf{(B1)}} or {\normalfont\textbf{(B2)}}. Then, the local time process $(L_{\tau_a^+}^{a-z}(X),z\geq 0)$ is infinitely divisible and its L\'evy measure $\nu^{(a)}$ is given by $$\nu^{(a)}(d\omega)=\int_0^a \overline{N}(L_\zeta^{\cdot-s} 1_{\{\cdot-s >0\}} \in d\omega)ds,$$ where for each $s$, $\overline{N}(L_\zeta^{\cdot-s} 1_{\{\cdot-s >0\}} \in d\omega)$ is the image of $\overline{N}$ under the map that assigns to each excursion its local time process shifted by $s$. Moreover, local times admit the representation
\begin{equation}
L_{\tau_a^+}^{a-z}(X)=\begin{cases}
\int_0^z \int_{D(0,\infty)} \ell (z-s) K(ds,d\ell), & z\in [0,a] \\
\int_0^a \int_{D(0,\infty)} \ell (z-s) K(ds,d\ell), & z \geq a.
\end{cases}
\end{equation}
where $K$ is a Poisson random measure of intensity $ds\otimes M(d\ell)$, $M$ being the image of $\overline{N}$ under the map that associates an excursion $\mathbf{e}$ to its local time process up to its lifetime $\zeta$: $\mathbf{e}\mapsto \left(L_\zeta ^r (\mathbf{e}),r\in \mathbb{R} \right)$.
\end{theorem}

This representation is similar to that in \cite[Ch.6]{li2020continuous} for CB processes with linear immigration but, as in the case of $\tau(c)$, $M$ cannot be a Kuznetsov measure. We also include here a description that reinforces the intuition on why local times, up to $\tau_a^+$, for instance, have a branching property (see also Figure~\ref{mplt} below). The local time of $X$ at each level can be decomposed as the sum of the local time contributions of these excursions to that level. We view the excursions from level $a$ downwards and think of them as individuals immigrating. The linear behavior of immigration comes from the fact that the supremum is linear on the local time scale. Since $S$ takes values on $[0,a]$, there is no extra immigrants for negative levels, which explains the difference in the representation in Theorem \ref{PRtauA} for levels in $[0,a]$ and levels in $(-\infty,0]$. For the branching part, if $0<x<y$, an excursion which has a contribution to the local time at $a-x$ will also have a contribution at $a-y$ if the excursion is deep enough. So, we can interpret the contribution of an excursion to $L_{\tau_a^+}^{a-y}(X)$ as a ``mass'' coming from that of $L_{\tau_a^+}^{a-x}(X)$.

\begin{figure}[htb]
\begin{center}
\includegraphics[width=0.6\textwidth]{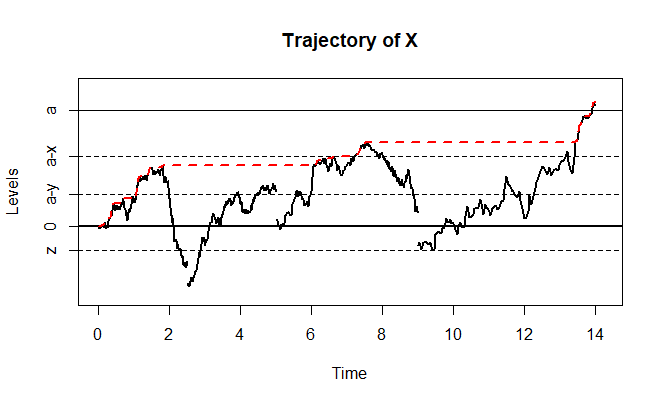}
\caption{\textit{\small A path of $X$ (bold line) and its supremum (dashed line) up to the time it surpasses level $a$.}}
\label{mplt}
\end{center}
\end{figure}
\FloatBarrier

We can get a better understanding of the L\'evy measure $\nu^{(a)}$ by studying the law of local times under $\overline{N}$. This turns out to be difficult if one wants to consider all the levels but it is tractable when considering a positive reference level $x>0$, which can be made arbitrarily small. 

Given $x>0$, excursions with height $H<x$ have local time equal to zero at levels $y \geq x$. Hence, we will restrict ourselves to those excursions satisfying $H>x$. On the event $H>x$, we know that, by the absence of negative jumps under $\overline{N}$, the first hitting time $T_x$ occurs before the lifetime $\zeta$. Then, by the strong Markov property and the additivity of local times we can write $$L_\zeta ^y = L_{T_x}^y + L_\zeta ^{y} \circ \theta_{T_x}, \quad y>x,$$ which splits the local time process for levels $y>x$ into two independent components: the information before hitting $x$ and the information after, the latter being able to be written in terms of the excursions away from $x$. Both terms can be described in a Poissonian way, similar to that in previous section, as we will see next.

We begin with $(L_{T_x}^y,y>x)$. We only need to consider the event $\tau_x^+<T_x<\zeta$, since on the event $\tau_x^+=T_x$ an excursion has local time zero at levels bigger than $x$. Therefore, the first overshoot of the excursion at $x$ will play a key role (see Lemma~\ref{ovN1} for its law under $\overline{N}$) and conditionally on it, we can use a similar decomposition of the path into excursions away from the infimum as in the previous section. This yields the following theorem.

\begin{theorem} \label{PR1Tx}
For any measurable and non-negative functional $F$, we have that
\begin{align*}
\overline{N}(F(L_{T_x}^y,y>x),H>x)&= \int_0^\infty \overline{N}(\mathcal{O}_x \in db, \tau_x^+<T_x<\zeta) \widehat{\kappa}(b, F(L_\zeta^{y-x},y>x)),
\end{align*}
with $\widehat{\kappa}$ as in \eqref{kaphat}. In particular, for any $f:(x,\infty) \to \mathbb{R}_+$ measurable and bounded,
\begin{align*}
&\overline{N} \left(\exp\left\{-\int_x^\infty f(y) L_{T_x}^y dy \right\} , H>x \right)\\
& = \frac1{c_-}\int_0^x d\ell \int_{(0,\infty)}  \widehat{\Pi}(db+\ell) W(x-\ell) \left(\frac{W^{\prime}(x)}{W(x)}-\frac{W^{\prime}(x-\ell)}{W(x-\ell)}\right) \mathcal{W}_{f_{x,b}}(b),
\end{align*}
where $\widehat{\Pi}$ is given by $\widehat{\Pi}(A)=\Pi(-A)$, $f_{x,b}(z):=f(x+b-z), \ z<b$ and $\mathcal{W}_{f_{x,b}}$ as in \eqref{newScaleF}.
\end{theorem}

Now focus on $(L_{\zeta}^y \circ \theta_{T_x},y>x)$. Starting from time $T_x$, we can decompose the path into excursions away from $x$. This bears some similarities with the situation of the second Ray-Knight Theorem, in which we had ``$c$'' excursions away from 0, but in this case, we have as many excursions away from $x$ as the local time $L_\zeta^x$, which can be proved that is exponentially distributed with parameter $q_x^0:=\widehat{N}_x (\tau_0^- <\zeta)$. This and the distribution of the overshoots of the excursions away from $x$ (see Lemma~\ref{ovNx}) lead to the following theorem. 

\begin{theorem} \label{PR1zeta}
For any measurable and non-negative function $f$, we have that
\begin{align} \label{brP}
&\overline{N}\left(\exp\left\{-\int_x^\infty f(y)L_{\zeta}^y \circ \theta_{T_x} dy  \right\},H>x \right) \nonumber \\
&= \overline{N}(H>x) \widehat{\mathbb{E}}_x^0 \left[\exp\left\{-L_\zeta^x \widehat{N}_x \left(1-e^{-\int_x^\infty f(y)L_\zeta^y dy}, \tau_0^- >\zeta \right) \right\} \right].
\end{align}
The term under $\widehat{N}_x$ can be further decomposed into $$\widehat{N}_x \left(1-e^{-\int_x^\infty f(y)L_\zeta^y dy}, \tau_0^- >\zeta \right)=\widehat{N}_x \left(1-e^{-\int_x^\infty f(y)L_\zeta^y dy}, \mathcal{E}_+^x \right)+\widehat{N}_x \left(1-e^{-\int_x^\infty f(y)L_\zeta^y dy}, \tau_0^- >\zeta, \mathcal{E}_\pm^x \right),$$ where $\mathcal{E}_+^x$ is the set of excursions which are completely above $x$ and $\mathcal{E}_\pm^x$ are the excursions away from $x$ which start negative and then jump above $x$. For this latter set, we have the following expression in terms of the law of the overshoot:
\begin{align*}
&\widehat{N}_x \left(1-e^{-\int_x^\infty f(y)L_\zeta^y dy}, \tau_0^- >\zeta, \mathcal{E}_\pm^x \right)\\
& \quad =\int_{(0,\infty)} \widehat{N}_x(\mathcal{O}_x \in db, \tau_0^- >\zeta, \mathcal{E}_\pm ^x) \left[ 1-\exp\left\{-\int_0^b ds \overline{N} \left(1-e^{-\int_{x}^\infty f(y) L_\zeta ^{y-x-s} 1_{\{y-s>x \}}} \right) \right\}\right] \\
& \quad =\int_{(0,\infty)} \widehat{N}_x(\mathcal{O}_x \in db, \tau_0^- >\zeta, \mathcal{E}_\pm ^x) \left[ 1-\mathcal{W}_{f_{x,b}}(b)\right],
\end{align*}
where $f_{x,b}(z)=f(x+b-z), \ z<b$ and $\mathcal{W}_{f_{x,b}}$ as in \eqref{newScaleF}.
\end{theorem}

\begin{rem}
Observe that identity \eqref{brP} is reminiscent of the branching property. Indeed, we can rewrite it as $$\frac{\overline{N}\left(\exp\left\{-\int_x^\infty f(y)L_{\zeta}^y \circ \theta_{T_x} dy  \right\},H>x \right)}{\overline{N}(H>x)}=  \widehat{\mathbb{E}}_x^0 \left[\exp\left\{-L_\zeta^x \widehat{N}_x \left(1-e^{-\int_x^\infty f(y)L_\zeta^y dy}, \tau_0^- >\zeta \right) \right\} \right].$$ The left hand side can be seen as an expected value conditioned to $H>x$, which is equivalent to having a positive amount of local time at $x$. And on the right hand side, we have that the contribution to the local time of each level $y$ is coming from the $L_\zeta^x$ ``individuals'' present at level $x$. In this case, the term corresponding to $\widehat{N}_x$ can be interpreted as the corresponding cumulant. Actually, if we consider a single level $y>x$, the Proposition~\ref{ovN3} below provides an explicit expression for this cumulant.
\end{rem}

It turns out that, without splitting the information at time $T_x$ for local times of levels bigger than $x$, a similar branching property is satisfied. Actually, we can consider a functional of all levels above $x$ as before or just consider a finite set of points, as states the following proposition. 

\begin{proposition} \label{Prop1}
Let $x>0$ and $f:(x,\infty) \to \mathbb{R}_+$ a measurable and bounded function. Denote by $$u_{x}(f):=\widehat{N}_x\left(1-e^{-\int_x^\infty f(y) L_\zeta^y dy},\tau_0^- >\zeta\right)={N}_{0}\left(1-e^{-\int_x^\infty f(y) L_\zeta^{x-y} dy},\tau_{x}^{+} >\zeta\right).$$ Then, $$\overline{N} \left(1-e^{-\lambda L_\zeta ^x-\int_x^\infty f(y) L_\zeta^y dy} \right)=\overline{N} \left(1-e^{-(\lambda+u_{x}(f)) L_{\zeta}^x - \int_x^\infty f(y) L_{T_x}^y dy} \right),\quad \lambda \geq 0.$$ Alternatively, if $0<x<y_1,\dots,y_n$ are $n$ distinct points and $$u_{x;y_1,\dots,y_n}(\beta_1,\dots,\beta_n):= \widehat{N}_x\left(1-e^{-\beta_1 L_\zeta ^{y_1}-\cdots - \beta_n L_\zeta^{y_n}},\tau_0^- >\zeta\right),$$ then, for any $\lambda,\beta_1,\dots,\beta_n \geq 0$, $$\overline{N} \left(1-e^{-\lambda L_\zeta ^x-\beta_1 L_\zeta ^{y_1}-\cdots - \beta_n L_\zeta^{y_n}} \right)=\overline{N} \left(1-e^{-(\lambda+u_{x;y_1,\dots,y_n}(\beta_1,\dots,\beta_n)) L_{\zeta}^x -\beta_1 L_{T_x}^{y_1}-\cdots -\beta_n L_{T_x}^{y_n} } \right).$$
\end{proposition}

We end this section with explicit expressions of the law of the local time of a single point under $\overline{N}$ and $\widehat{N}_x$, which can be given in terms of scale functions.

\begin{proposition}\label{ovN2}
Let $$u_y(\lambda):=\overline{N}\left(1-e^{-\lambda L^{y}_{\zeta}}\right),\qquad y> 0, \ \lambda\geq 0.$$
This quantity can be expressed in terms of $W$ and the constant $c_+$ as follows
\begin{equation*}
u_y(\lambda)=\frac{\lambda W'(y)}{c_+ +\lambda W(y)},\qquad \lambda\geq 0, y\geq 0.
\end{equation*}
\end{proposition}

\begin{proposition}\label{ovN3}
Let us define, for any $x>0$,
$$v_{x,y}(\lambda):=\widehat{N}_{x}\left(\left(1-e^{-\lambda L^{y}_{\zeta}}\right)1_{\{\tau^{-}_{0}>\zeta\}}\right),\qquad y\geq x, \ \lambda\geq 0.$$ Then,
\begin{align*}
v_{x,y}(\lambda)&=\widehat{N}_x^0(H>y)\left(\frac{\lambda W(y-x)}{c_+ +\lambda W(y-x)}\right),\qquad \lambda\geq 0, y\geq x, 
\end{align*}
where $\widehat{N}_x^0(\cdot):=\widehat{N}_x\left(\cdot;\tau_0^->\zeta \right)$. Moreover,
\begin{align*}
&\widehat{N}_x^0(H>y)= \frac{c_-\sigma^2}{2}\frac{W'(y-x)}{W(y-x)} \\
&\ +\int_0^x \left( 1-c_+ + \frac{c_+ W(x-z)}{W(x)} \right) \left(\Pi(-\infty,-z)-\int_{-z-(y-x)}^{-z} \frac{W(u+z+y-x)}{W(y-x)}\Pi(du)\right)dz,
\end{align*}
where $c_+$ and $c_-$ are the same constants depending on the normalization of the local times at the supremum and infimum as before, respectively.
\end{proposition}

\section{Decomposition of L\'evy measures}\label{decomp}

Recall that a non-negative infinitely divisible process $\psi=(\psi_x,x\in E)$ is characterized by its L\'evy measure $\mu$. Given a reference state $h\in E$, Eisenbaum \cite{eisenbaum2019decompositions} also describes in their Theorem 1.2 a way to decompose $\mu$ into $\mu_h + \overline{\mu}_h$. The measure $\mu_h(d\omega)=\mu(d\omega) \big|_{\left\{\omega(h)=0 \right\}}$ is the L\'evy measure of the process $\psi$ conditioned to $\psi_h=0$ and hence, the measure $\overline{\mu}_h(d\omega)=\mu(d\omega) \big|_{\left\{\omega(h)>0 \right\}}$ corresponds to the information of the process between succesive visits to state $h$. 

As an application, we provide some information on this decomposition of the L\'evy measures $\mu^{(c)}$, relative to $(L_{\tau(c)}^y(X),y\in \mathbb{R})$ and $\nu^{(a)}$, corresponding to $(L_{\tau_a^+}^{a-z}(X),z\geq 0)$.

Consider first the measure $\mu^{(c)}$, which is given by $\mu^{(c)}(d\omega)=cN_0(L_\zeta^\cdot \in d\omega)$. For a positive level $h$, we can identify one of the components of the corresponding decomposition in terms of an exit problem (and hence in terms of scale functions), as can be seen in the next proposition.

\begin{proposition} \label{muc_h}
Let $\mu^{(c)}$ be the L\'evy measure of the process $(L_{\tau(c)}^y (X), y \in \mathbb{R})$ and $h>0$ fixed. Decompose $$\mu^{(c)}(d\omega)=\mu_h^{(c)}(d\omega)+\overline{\mu}^{(c)}_h(d\omega)=\mu^{(c)}(d\omega)\big|_{\{ \omega(h)=0 \}}+\mu^{(c)}(d\omega)\big|_{\{\omega(h)>0\}}.$$ Then,
\begin{align*}
\mu^{(c)}_h(d\omega)=cN_0 \left[L_\zeta^\cdot \in d\omega, \tau_h^+>\zeta \right]&= cN_0 \left[L_\zeta^\cdot \in d\omega, \tau_h^+>\zeta,\mathcal{E}_+ \right] + cN_0 \left[L_\zeta^\cdot \in d\omega,\mathcal{E}_- \right]\\
&\ + c\int_{(0,h)}\int_{(-\infty,0)} db\Pi(du-b) \widehat{\mathbb{E}}_{b}^{0} \otimes \mathbb{E}_{u}^{0}(L_\zeta^\cdot \in d\omega,\tau_h^+ >\zeta).
\end{align*}
In particular, if $X$ does not have a Brownian component $(\Sigma=0)$, then for any $f:\mathbb{R}_+ \to \mathbb{R}$ measurable and bounded, $$\int_{\mathbb{R}_+^{\mathbb{R}}}\left(1-e^{-\int_0^\infty f(x)\omega(x)dx} \right) \mu_h^{(c)}(d\omega) = c\int_0^h db \Pi(-\infty,-b) \frac{W_{\widehat{f}}(-b,-h)}{W_{\widehat{f}}(0,-h)},$$ where $\widehat{f}(s)=f(-s)$ and $W_{\widehat{f}}$ as in \eqref{scalef}.
\end{proposition}

\begin{proof}[Proof of Proposition~\ref{muc_h}]
In this case, we know that $\mu^{(c)}(d\omega)=cN_0(L_\zeta^{\cdot} \in d\omega)$. Hence, for $h>0$, we have that $L_{\zeta}^h=0$ under $N_0$ if and only if the excursion $\mathbf{e}$ away from $0$ does not reach level $h$. This means that positive excursions must not have height bigger than $h$ and that the paths $\underleftarrow{\mathbf{e}}$ for mixed excursions have overshoots less than $h$ and from this point, the reversed path must exit the interval $[0,h]$ by below. So we identify in this case $\mu^{(c)}_h$ as
\begin{align*}
\mu^{(c)}_h(d\omega)=cN_0 \left[L_\zeta^\cdot \in d\omega, \tau_h^+>\zeta \right] &= cN_0 \left[L_\zeta^\cdot \in d\omega, \tau_h^+>\zeta,\mathcal{E}_+ \right] + cN_0 \left[L_\zeta^\cdot \in d\omega,\mathcal{E}_- \right]\\
&+ c\int_{(0,h)}\int_{(-\infty,0)} db\Pi(du-b) \widehat{\mathbb{E}}_{b}^{0} \otimes \mathbb{E}_{u}^{0}(L_\zeta^\cdot \in d\omega,\tau_h^+ >\zeta).
\end{align*}

In particular, if one only considers positive levels the last term simplifies to $\widehat{\mathbb{E}}_{b} (L_\zeta^\cdot \in d\omega,\tau_h^+ >\tau_0^-)$, which is the law of local times of $\widehat{X}$ issued from $b$ and seen up to first time it exits $[0,h]$ from above. An expression for Laplace transforms of this exit problem is given in terms generalized scale functions (see \eqref{scalef}). Indeed, if $f$ is measurable and bounded and has support on $\mathbb{R}_+$,
\begin{align*}
\widehat{\mathbb{E}}^0_{b} \left(e^{-\int_0^\infty f(x)L_{\zeta}^x dx} ,\tau_h^+ > \zeta \right) &= \widehat{\mathbb{E}}_{b} \left(e^{-\int_0^\infty f(x)L_{\tau_0^-}^x dx} ,\tau_h^+ > \tau_0^- \right) \\
&= \widehat{\mathbb{E}}_{b} \left(e^{-\int_0^{\tau_0^-} f(X_s)ds} ,\tau_h^+ > \tau_0^- \right) \\
&= \mathbb{E}_{-b} \left(e^{-\int_0^{\tau_0^+} \widehat{f}(X_s)ds} ,\tau_{-h}^- > \tau_0^+ \right) \\
&= \frac{W_{\widehat{f}}(-b,-h)}{W_{\widehat{f}}(0,-h)}.
\end{align*}
Hence, if $X$ does not have a Brownian component, we conclude that $$\int_{\mathbb{R}_+^{\mathbb{R}}}\left(1-e^{-\int_0^\infty f(x)\omega(x)dx} \right) \nu_h(d\omega) = c\int_0^h db \Pi(-\infty,-b) \frac{W_{\widehat{f}}(-b,-h)}{W_{\widehat{f}}(0,-h)}.$$    
\end{proof}

Now consider the L\'evy measure $\nu^{(a)}$, which is given by $\nu^{(a)}(d\omega)=\int_0^a \overline{N}(L_\zeta^{\cdot-s}1_{\{\cdot-s >0 \}} \in d\omega)ds$. It turns out that, for a fixed $h>0$, we can make use of the results of Theorems ~\ref{PR1Tx} and ~\ref{PR1zeta} to describe certain functionals of both $\nu^{(a)}_h$ and $\overline{\nu}^{(a)}_h$.

\begin{proposition}\label{nua_h}
Let $\nu^{(a)}$ be the L\'evy measure of the process $(L_{\tau_a^+}^{a-z} (X), z \geq 0)$ and $h>0$ fixed. Decompose $$\nu^{(a)}(d\omega)=\nu_h^{(a)}(d\omega)+\overline{\nu}^{(a)}_h(d\omega)=\nu^{(a)}(d\omega)\big|_{\{ \omega(h)=0 \}}+\nu^{(a)}(d\omega)\big|_{\{\omega(h)>0\}}.$$ Then,
\begin{align*}
\nu_h^{(a)}(d\omega) = \int_{0}^{a\wedge h} ds \overline{N}(L_\zeta ^{\cdot-s} 1_{\{\cdot-s >0 \}}\in d\omega,\tau^+_{h-s}>\zeta) + \int_{a\wedge h}^{a} ds \overline{N}(L_\zeta ^{\cdot-s} 1_{\{\cdot-s >0\}} \in d\omega). 
\end{align*}
Moreover, if $F$ is any non-negative, measurable and bounded functional,
\begin{align*}
 \nu_h^{(a)}(F(\omega_y,y\geq h)) &= \int_0^a ds \overline{N} (F(L_{T_h}^{y-s},y-s > h), H>h) \\
 &= \int_0^a ds \int_{(0,\infty)} \overline{N}(\mathcal{O}_h \in db, \tau_h^+<T_h<\zeta) \widehat{\kappa}(b, F(L_\zeta^{y-s-h},y-s>h)),
\end{align*}
with $\widehat{\kappa}$ as in \eqref{kaphat}, and on the other hand,
\begin{align*}
\overline{\nu}_h^{(a)} (F(\omega_y,y>h)) = \int_0^a ds \overline{N} (F(L_\zeta ^{y-s} \circ \theta_{T_h}, y-s>h),H>h).    
\end{align*}
\end{proposition}

\begin{proof}[Proof of Proposition~\ref{nua_h}]
In this case one has that $L_{\zeta}^{h-s}=0$ if and only if either $s \geq h$ or $s<h$ and the excursion away from the supremum does not reach level $h-s$. Therefore, we obtain
\begin{align*}
\nu_h^{(a)}(d\omega) &= \nu^{(a)}(d\omega)\bigg|_{\{\omega(h)=0 \}} \\
&= \int_{0}^{a\wedge h} ds \overline{N}(L_\zeta ^{\cdot-s} 1_{\{\cdot-s >0 \}}\in d\omega,\tau^+_{h-s}>\zeta) + \int_{a\wedge h}^{a} ds \overline{N}(L_\zeta ^{\cdot-s} 1_{\{\cdot-s >0\}} \in d\omega), 
\end{align*}
where the last term disappears if $h\geq a$.
Recall that Theorem~\ref{PR1Tx} gives an expression for the law of local times of levels bigger than a reference level $h>0$ under $\overline{N}$. Since the process $(L_{T_h}^y,y>h)$ codes the information of the local times of levels bigger than $h$ previous to its first hitting time, this implies we can actually use that result to provide information on $\nu_h^{(a)}$. Therefore, if $F$ is any non-negative, measurable and bounded functional,
\begin{align*}
 \nu_h^{(a)}(F(\omega_y,y\geq h)) &= \int_0^a ds \overline{N} (F(L_{T_h}^{y-s},y-s > h), H>h) \\
 &= \int_0^a ds \int_{(0,\infty)} \overline{N}(\mathcal{O}_h \in db, \tau_h^+<T_h<\zeta) \widehat{\kappa}(b, F(L_\zeta^{y-s-h},y-s>h)).
\end{align*}

Similarly, in Theorem~\ref{PR1zeta} the process $(L_{\zeta}^y \circ \theta_{T_h},y>h)$ encodes the information of the excursion from the first to the last visit to $h$, which implies we can recover information on the measure $\overline{\nu}_h^{(a)}$. Then, for any measurable, bounded and non-negative functional $F$, 
\begin{align*}
\overline{\nu}_h^{(a)} (F(\omega_y,y>h)) = \int_0^a ds \overline{N} (F(L_\zeta ^{y-s} \circ \theta_{T_h}, y-s>h),H>h).    
\end{align*}    
\end{proof}

In particular, if the functional on the above proposition is of the form $F((\omega_y,y>h))=\exp\left\{-\int_h^\infty f(y)\omega_y dy \right\}$, we can write $$\nu_h^{(a)}(F(\omega_y,y\geq h))= \frac1{c_-}\int_0^a ds \int_0^x d\ell \int_{(0,\infty)} \widehat{\Pi}(db+\ell) W(x-\ell) \left(\frac{W^{\prime}(x)}{W(x)}-\frac{W^{\prime}(x-\ell)}{W(x-\ell)}\right) \mathcal{W}_{f_{x,b,s}}(b),$$ where $f_{x,b,s}(z):=f(x+b+s-z), \ z<s+b$. Similarly, $$\overline{\nu}_h^{(a)} (F(\omega_y,y>h))=\overline{N}(H>h) \int_0^a ds \widehat{\mathbb{E}}_h^0 \left[\exp\left\{-L_\zeta^h \widehat{N}_x \left(1-e^{-\int_h^\infty f(s+y)L_\zeta^y dy}, \tau_0^- >\zeta \right) \right\} \right].$$

\section{Auxiliary results} \label{auxR}

The next lemma provides the joint law of the overshoot and undershoot of an excursion under the measure $N_0$.

\begin{lemma} \label{ouN0}
Let $(\mathcal{U}_0(\mathbf{e}),\mathcal{O}_0(\mathbf{e}))=(\mathbf{e}(\tau_0^-),\mathbf{e}(\tau_0^--),)$ be the undershoot and overshoot of an excursion $\mathbf{e}$ away from zero. Then,
\begin{equation}\label{jointOU}
N_0 \left(h(\mathcal{U}_0,\mathcal{O}_0),\mathcal{E}_\pm \right)=\int_0^\infty dz \int_{(-\infty,0)} \Pi(dy) h(z+y,z)1_{\{z+y<0 \}},
\end{equation}
for any measurable and bounded function $h:(-\infty,0)\times (0,\infty) \to \mathbb{R}_+$. Said otherwise,
\begin{equation}
N_0(\mathcal{U}_0 \in dy,\mathcal{O}_0 \in dz, \mathcal{E}_\pm)=dz \Pi(dy-z) 1_{\{z>0\}} 1_{\{y<0\}}.
\end{equation}
Taking marginals in the above expression we get that, for any $f:(-\infty,0)\to \mathbb{R}_+$ and $g:(0,\infty)\to \mathbb{R}_+$ measurable and bounded,
\begin{equation*}
N_0 \left(f(\mathcal{U}_0),\mathcal{E}_\pm \right)=\int_0^\infty dz \int_{(-\infty,0)} \Pi(dy) f(z+y)1_{\{z+y<0\}}=\int_0^\infty dz \int_{(-\infty,z)}\Pi(dy-z) f(y)1_{\{y<0 \}}
\end{equation*}
and
\begin{equation*}
N_0 \left(g(\mathcal{O}_0),\mathcal{E}_\pm \right)=\int_0^\infty dz g(z) \Pi(-\infty,-z).
\end{equation*}
\end{lemma}

\begin{proof}[Proof of Lemma~\ref{ouN0}]
Recall that $\mathcal{E}_\pm$ corresponds to the set of excursions for which $\{0<\tau_0^-<\zeta \}$. Observe that this condition is fulfilled if and only if there exists an $s \in (0,\zeta)$ such that $\Delta_s:=\mathbf{e}(s)-\mathbf{e}(s-)<0$, $\mathbf{e}(s-)>0$ and $\mathbf{e}(s-)+\Delta_s <0$. Let us define $$G_s(y)=h(\mathbf{e}(s-)+y,\mathbf{e}(s-)) 1_{\{y<0 \}} 1_{\{\mathbf{e}(s-)>0 \}} 1_{\{ \mathbf{e}(s-)+y<0 \}} 1_{\{s<\zeta \}}.$$ Then, we have the following identity $$h(\mathbf{e}(\tau_0^-),\mathbf{e}(\tau_0^--))1_{\{0<\tau_0^-<\zeta \}}=\sum_{0<s<\infty} G_s(\Delta_s) 1_{\{\Delta_s \neq 0 \}}.$$ Indeed, by the observation above the indicator on the left hand side is non-zero if and only if there exists an $s\in (0,\infty)$ satisfying the previous conditions. Moreover, by the absence of positive jumps that $s$ is unique and coincides with $\tau_0^-$. Therefore, we can use the compensation formula under $N_0$ (see for instance \cite[Eq.(18)]{pardo2018excursion}) to obtain
\begin{align*}
 &N_0 \left(h(\mathbf{e}(\tau_0^-),\mathbf{e}(\tau_0^--))1_{\{0<\tau_0^-<\zeta \}} \right) \\
 &\qquad = N_0 \left(\int_0^\infty ds \int_{(-\infty,0)} \Pi(dy) G_s(y) \right)\\
 &\qquad = N_0 \left(\int_0^\infty ds \int_{(-\infty,0)} \Pi(dy) h(\mathbf{e}(s-)+y,\mathbf{e}(s-)) 1_{\{y<0 \}} 1_{\{\mathbf{e}(s-)>0 \}} 1_{\{ \mathbf{e}(s-)+y<0 \}} 1_{\{s<\zeta \}} \right) \\
 &\qquad = N_0 \left(\int_0^\zeta ds 1_{\{\mathbf{e}(s-)>0 \}} \int_{(-\infty,0)} \Pi(dy) h(\mathbf{e}(s-)+y,\mathbf{e}(s-))  1_{\{ \mathbf{e}(s-)+y<0 \}} \right).
\end{align*}
Since the set of times $\{s:\mathbf{e}(s-)\neq \mathbf{e}(s) \}$ in which the excursion is discontinuous has Lebesgue measure zero, we can replace $\mathbf{e}(s-)$ by $\mathbf{e}(s)$. Therefore,
\begin{align*}
 &\qquad = N_0 \left(\int_0^\zeta ds 1_{\{\mathbf{e}(s)>0 \}} \int_{(-\infty,0)} \Pi(dy) h(\mathbf{e}(s)+y,\mathbf{e}(s))  1_{\{ \mathbf{e}(s)+y<0 \}} \right) \\
 &\qquad = N_0 \left(\int_0^\zeta ds 1_{\{\mathbf{e}(s)>0 \}} \widetilde{h}(\mathbf{e}(s)) \right),
\end{align*}
where $\widetilde{h}(z)=\int_{(-\infty,0)}\Pi(dy) h(z+y,z)1_{\{ z+y<0 \}}$. Then, using  the fact that $\Phi(0)=0$ because of assumptions \textbf{(B1)} or \textbf{(B2)} and using identity (20) from \cite{pardo2018excursion} with $\lambda \downarrow 0$, we conclude that
\begin{align*}
 N_0 \left(h(\mathbf{e}(\tau_0^-),\mathbf{e}(\tau_0^--))1_{\{0<\tau_0^-<\zeta \}} \right)&= N_0 \left(\int_0^\zeta ds 1_{\{s<\tau_0^- \}} \widetilde{h}(\mathbf{e}(s)) \right)\\
 &= \int_0^\infty dz \widetilde{h}(z) \\
 &= \int_0^\infty dz \int_{(-\infty,0)} \Pi(dy) h(z+y,z)1_{\{z+y<0\}} \\
 &= \int_0^\infty dz \int_{(-\infty,z)} \Pi(dy-z)h(y,z) 1_{\{y<0 \}} \\
 &= \int_0^\infty dz \int_{(-\infty,0)} \Pi(dy-z)h(y,z).
 \end{align*}
The rest of the proof follows readily.
\end{proof}

The following lemma expresses the law of the overshoot with respect to a positive reference level under the measure $\overline{N}$.

\begin{lemma}\label{ovN1}
There exists a constant $c_- \in (0,\infty)$, which depends on the normalization of the local time at the infimum, such that, for any $x>0$ and any measurable and bounded function $f:\mathbb{R}^2 \to \mathbb{R}_+$,
\begin{equation*}
\overline{N}\left(\mathbf{e}(\tau^{+}_{x})=x, \tau^{+}_{x}<\zeta\right)=-\frac{\sigma^2}{2c_-}W^{\prime\prime}(x),
\end{equation*}
and
\begin{align*}
\overline{N}\left(f(\mathbf{e}(\tau^{+}_{x}-),\mathbf{e}(\tau^{+}_{x})); X_{\tau_x^+}>x, \tau^{+}_{x}<\zeta\right)&=\frac1{c_-}\int^{x}_{0}dlW(x-l)\left(\frac{W^{\prime}(x)}{W(x)}-\frac{W^{\prime}(x-l)}{W(x-l)}\right) \\
& \quad  \times \int^{0}_{-\infty}\Pi(dy)f(x-\ell,x-\ell-y )1_{\{0>\ell+y\}}.
\end{align*}

In particular, taking $f$ to be a function of only the overshoot, $f(\mathbf{e}(\tau_x^+-), \mathbf{e}(\tau_x^+))=g(\mathbf{e}(\tau_x^+)-x)$, and with the notation $\widehat{\Pi}(A):= \Pi(-A)$, we have that 
\begin{align*}
&\overline{N}\left(g(\mathbf{e}(\tau_x^+)-x);  T_x<\tau^{+}_{x}<\zeta\right) \\
&\quad =\frac1{c_-} \int_0^x d\ell W(x-\ell) \left(\frac{W^{\prime}(x)}{W(x)}-\frac{W^{\prime}(x-\ell)}{W(x-\ell)}\right) \int_{0}^{\infty}\widehat{\Pi}(dy)g(y-\ell)1_{\{y>\ell\}} \\
&\quad = \frac1{c_-} \int_0^x d\ell W(x-\ell) \left(\frac{W^{\prime}(x)}{W(x)}-\frac{W^{\prime}(x-\ell)}{W(x-\ell)}\right) \int_{0}^{\infty}\widehat{\Pi}(dy+\ell)g(y)1_{\{y>0\}},
\end{align*}
which provides the law of the overshoot with respect to $x$ under $\overline{N}$.
\end{lemma}

\begin{proof}[Proof of Lemma \ref{ovN1}]
We will use here a result from L. Chaumont and R. Doney in \cite{chaumont2005levy} which tells us a way to approximate $\overline{N}$ as a certain limit involving $\mathbb{E}$. Indeed, if $g_\beta(z)=(1-e^{z\Phi(\beta)}{\Phi(\beta)})$ $z\in \mathbb{R}$, using that in our case $\Phi(0)=0$ we have that $$\lim_{\beta \downarrow 0} g_\beta (z) := g(z)= -z.$$ Their result states that there exists a constant $c_- \in (0,\infty)$ such that $$\lim_{z\to 0-} \frac1{g(z)} \mathbb{E}_z (F,t<\tau_0^+) = c_- \widehat{\overline{N}} (F,t<\zeta),$$ for any functional $F$ up to time $t$ and $\widehat{\overline{N}}$ being the excursion measure away from the supremum for $\widehat{X}$. This translates to $\overline{N}$ as $$\lim_{z\to 0-} \frac1{g(z)} \widehat{\mathbb{E}}_{-z} (F,t<\tau_0^-) = c_-\overline{N} (F,t<\zeta)$$ or equivalently, $$\lim_{z\to 0+} \frac1{z} \widehat{\mathbb{E}}_{z} (F,t<\tau_0^-) = c_-\overline{N} (F,t<\zeta).$$

The so-called Kesten's identity (see for instance \cite{kuznetsov2012theory}): $$\mathbb{P}_z \left(X_{\tau_0^-}=0, \tau_0^-=0 \right)=\frac{\sigma^2}{2} (W'(z)-\Phi(0)W(z)), \quad z>0,$$ will also be useful for the result.

We will use the above facts for the following computation. We can calculate, for $x>0$ 
\begin{align*}
c_- \overline{N}\left(\mathbf{e}(\tau^{+}_{x})=x, \tau^{+}_{x}<\zeta\right)&=\lim_{z\to 0+}\frac{1}{z}\widehat{\mathbb{E}}_{z}\left(X_{\tau^{+}_{x}}=x, \tau^{+}_{x}<\tau^{-}_{0}\right)\\
&=\lim_{z\to0+}\frac{1}{z}\mathbb{E}\left(X_{\tau^{-}_{z-x}}=z-x, \tau^{-}_{z-x}<\tau^{+}_{z}\right)\\
&=\lim_{z\to0+}\frac{1}{z}\left[\mathbb{E}\left(X_{\tau^{-}_{z-x}}=z-x, \tau^{-}_{z-x}<\infty\right) \right. \\
&\quad \quad \left. -\mathbb{E}\left(X_{\tau^{-}_{z-x}}=z-x, \tau^{+}_{z}<\tau^{-}_{z-x}\right)\right].
\end{align*}
Then, using Kesten's identity and the Markov property at $\tau_z^+$ we get
\begin{align*}
c_- \overline{N}\left(\mathbf{e}(\tau^{+}_{x})=x, \tau^{+}_{x}<\zeta\right)&=\lim_{z\to 0}\frac{1}{z}\left[\frac{\sigma^2}{2}\left(W^{\prime}(x-z)-\Phi(0)W(x-z) \right) \right. \\
& \quad \quad \left.- \mathbb{E}\left(\mathbb{P}_{z}\left(X_{\tau^{-}_{z-x}}=z-x, \tau^{-}_{z-x}<\infty \right), \tau^{+}_z<\infty \right)\right]\\
&=\lim_{z\to 0}\frac{1}{z}\left[\frac{\sigma^2}{2}\left(W^{\prime}(x-z)-\Phi(0)W(x-z)\right) \right. \\
& \quad \quad \left. -\left(e^{-\Phi(0)z}\left[\frac{\sigma^2}{2}\left(W^{\prime}(x)-\Phi(0)W(x)\right)\right]\right)\right]\\
&= \lim_{z\to 0}\frac{1}{z}\left[\frac{\sigma^2}{2}W^{\prime}(x-z)- \frac{\sigma^2}{2} W'(x) \right] \\
&=-\frac{\sigma^2}{2}W^{\prime\prime}(x),
\end{align*}
where for the second term in the fourth equality we used the strong Markov property at time $\tau_z^+$. We conclude that $$\overline{N}\left(\mathbf{e}(\tau^{+}_{x})=x, \tau^{+}_{x}<\zeta\right) = -\frac{\sigma^2}{2c_-} W''(x).$$

For the other identity, consider $g:\mathbb{R} \to \mathbb{R}_+$ measurable and bounded. We will use the fact that the potential of $\widehat{X}$ killed when it exits the closed interval $[0,x]$ for the first time is given in terms of scale functions, as can be seen in \cite[Theorem 8.7]{kyprianou2014fluctuations}. Indeed, we have
\begin{align*}
\widehat{U}^0 g(z):&=\widehat{E}_z\left(\int^{\tau_x^+ \wedge \tau^{-}_{0}}_{0}dtg(X_t)\right)=E_{-z}\left(\int^{\tau_{-x}^- \wedge \tau^{+}_{0}}_{0} dt g(-X_t)\right) \\
&= E_{x-z}\left(\int^{\tau_{0}^- \wedge \tau^{+}_{x}}_{0} dt g(-X_t+x)\right) \\
&=\int_0^x dv \left(\frac{W(x-z)}{W(x)}W(x-v)-W(x-z-v) \right) g(-v+x) \\
&= \int_0^x dv \left(\frac{W(x-z)}{W(x)}W(v)-W(v-z) \right) g(v).
\end{align*}

This and the previous limit result of Chaumont and Doney allows to give an expression for the potential up to $\tau_x^+$ under $\overline{N}$. Namely,
\begin{equation*}
\begin{split}
\overline{N}\left(\int^{\tau^{+}_{x}}_{0}dt g(\mathbf{e}(t)), \tau_x^+ <\zeta \right)&=\lim_{z\to 0+}\frac{1}{z}\left[\widehat{\mathbb{E}}_z\left(\int^{\tau^{+}_x\wedge \tau^{-}_0}_0dt g (X_t)\right)\right]\\
&=\lim_{z\to 0}\frac{1}{z}\int^{x}_{0}\left(\frac{W(x-z)}{W(x)}W(v)-W(v-z)\right)g(v)dv\\
&=\int^{x}_{0} dv\left(\frac{W(v)W^{\prime}(x)}{W(x)}-W^{\prime}(v)\right) g(v),
\end{split}
\end{equation*}
where we have used that 
\begin{align*}
\lim_{z\to 0} \frac1{z} \left(\frac{W(x-z)}{W(x)}W(v)-W(v-z)\right) &= \lim_{z\to 0} \frac1{z} \left(\frac{W(x-z)-W(x)+W(x)}{W(x)}W(v)-W(v-z)\right) \\
&= \lim_{z\to 0} \left(\frac{W(x-z)-W(x)}{z}\frac{W(v)}{W(x)} + \frac{W(x)}{zW(x)}W(v)- \frac{W(v-z)}{z} \right)\\
&= \lim_{z\to 0} \left(\frac{W(x-z)-W(x)}{z}\frac{W(v)}{W(x)} - \frac{W(v-z)-W(v)}{z} \right)\\
&= \frac{W(v)W^{\prime}(x)}{W(x)}-W^{\prime}(v).
\end{align*}

We will now use the compensation formula, as in the proof of Lemma~\ref{ouN0}, to compute the joint law $\overline{N}\left(f(\mathbf{e}(\tau^{+}_x-), \mathbf{e}(\tau^{+}_x)), \mathbf{e}(\tau^{+}_x)>x, \tau_x^+<\zeta \right)$ in terms of the jumps $\Delta_s$ of the excursion. Indeed, the event $\{\tau_x^+ <\zeta, \mathbf{e}(\tau_x^+)>x \}$ is equivalent to the existence of $s \in (0,\zeta)$ such that $\sup_{r\in (0,s)} \mathbf{e}(r) <x$ and $\mathbf{e}(s-)+\Delta_s>x$ (observe that such $s$ is unique by definition). Define then $$G_s(y)=f(\mathbf{e}(s-),\mathbf{e}(s-)+y) 1_{\{\sup_{(0,s)}\mathbf{e}<x  \}} 1_{\{\mathbf{e}(s-)+y >x \}} 1_{\{ s<\zeta \}}, \quad s>0,$$ to write $$f(\mathbf{e}(\tau^{+}_x-), \mathbf{e}(\tau^{+}_x)), 1_{\{\mathbf{e}(\tau^{+}_x)>x, \tau_x^+<\zeta \}}= \sum_{0<s<\infty} G_s(\Delta_s) 1_{\{\Delta_s >0 \}}.$$ Applying the compensation formula we get
\begin{align*}
&\overline{N} \left( f(\mathbf{e}(\tau^{+}_x-), \mathbf{e}(\tau^{+}_x)), 1_{\{\mathbf{e}(\tau^{+}_x)>x, \tau_x^+<\zeta \}} \right) = \overline{N} \left( \sum_{0<s<\infty} G_s(\Delta_s) 1_{\{\Delta_s >0 \}} \right) \\
& \quad = \overline{N} \left(\int_0^\infty ds \int_{(0,\infty)} \widehat{\Pi}(dy) G_s(y) \right) \\
& \quad = \overline{N} \left(\int_0^\infty ds \int_{(0,\infty)} \widehat{\Pi}(dy) f(\mathbf{e}(s-),\mathbf{e}(s-)+y) 1_{\{\sup_{(0,s)}\mathbf{e}<x  \}} 1_{\{\mathbf{e}(s-)+y>x \}} 1_{\{ s<\zeta \}} \right) \\
& \quad = \overline{N} \left(\int_0^\zeta ds 1_{\{\sup_{(0,s)}\mathbf{e}<x  \}}  \int_{(0,\infty)} \widehat{\Pi}(dy) f(\mathbf{e}(s),\mathbf{e}(s)+y) 1_{\{\mathbf{e}(s)+y>x \}} \right) \\
& \quad = \overline{N} \left(\int_0^\zeta ds 1_{\{\tau_x^+>s  \}}  \int_{(0,\infty)} \widehat{\Pi}(dy) f(\mathbf{e}(s),\mathbf{e}(s)+y) 1_{\{\mathbf{e}(s)+y>x \}} \right)  \\
& \quad = \overline{N} \left(\int_0^{\tau_x^+} ds g(\mathbf{e}(s)), \tau_x^+<\zeta \right)
\end{align*}
where $g(v)= \int_{(0,\infty)} \widehat{\Pi}(dy) f(v,v+y)1_{\{v+y>x \}}$. Therefore, applying the previously obtained formula for the potential under $\overline{N}$ we conclude
\begin{align*}
&\overline{N} \left( f(\mathbf{e}(\tau^{+}_x-), \mathbf{e}(\tau^{+}_x)), 1_{\{\mathbf{e}(\tau^{+}_x)>x, \tau_x^+<\zeta \}} \right) = \int_0^x dv \left(\frac{W'(x) W(v)}{W(x)}-W'(v) \right) g(v) \\
& \quad = \int_0^x dv \left(\frac{W'(x) W(v)}{W(x)}-W'(v) \right) \int_{(-\infty,0)} \Pi(dy) f(v,v-y)1_{ \{v-y>x \}} \\
& \quad = \int_0^x d\ell \left(\frac{W'(x) W(x-\ell)}{W(x)}-W'(x-\ell) \right) \int_{(-\infty,0)} \Pi(dy) f(x-\ell,x-\ell-y)1_{ \{x-\ell-y>x \}} \\
& \quad = \int_0^x d\ell W(x-\ell)\left(\frac{W'(x)}{W(x)}-\frac{W'(x-\ell)}{W(x-\ell)} \right) \int_{(-\infty,0)} \Pi(dy) f(x-\ell,x-\ell-y)1_{ \{0>\ell + y \}},
\end{align*}
completing the proof.
\end{proof}

This lemma gives the law of the overshoot but now with respect to the measure $\widehat{N}_x$.

\begin{lemma}\label{ovNx}
Let $g:\mathbb{R} \to \mathbb{R}_+$ measurable and bounded and denote by $\widehat{\Pi}(dy)=\Pi(-dy)$. Then, there exists a constant $c_+>0$, which depends on the normalization of the local time at the supremum, such that $$\widehat{N}_x \left(g(\mathcal{O}_x), \tau_0^- >\zeta, \mathcal{E}_\pm^x \right)=\int_0^x dz \left(1-c_+ + \frac{c_+ W(x-z)}{W(x)} \right) \int_{(0,\infty)} \widehat{\Pi}(dy+z) g(y).$$
\end{lemma}

\begin{proof}[Proof of Lemma~\ref{ovNx}]
By space invariance and duality, we have that
\begin{align*}
\widehat{N}_x \left(g(\mathcal{O}_x), \tau_0^- >\zeta, \mathcal{E}_\pm^x \right)&=  \widehat{N}_0 \left(g(\mathcal{O}_0), \tau_{-x}^- >\zeta, \mathcal{E}_\pm \right) \\
&= N_0 \left(g(-\mathcal{U}_0), \tau_{x}^+ >\zeta, \mathcal{E}_\pm \right),
\end{align*}
which is almost the law of the undershoot we have computed in Lemma~\ref{ouN0} but with the additional condition $\{\tau_x^+ >\zeta \}$, which means that the excursion prior to $\tau_0^-$ does not go above $x$. Let $h(u,b):=g(-u)$. Mimicking the computations in the proof of the op. cit. lemma but now considering $$G_s(y)=h(\mathbf{e}(s-)+y,\mathbf{e}(s-)) 1_{\{y<0 \}} 1_{\{ \mathbf{e}(s-) \in (0,x) \}} 1_{\{ \mathbf{e}(s-)+y <0 \}} 1_{\{s<\zeta \}} 1_{\{\sup_{r \in (0,s)} \mathbf{e}(r)<x \}}, $$ we get that
\begin{align*}
& N_0 \left(g(-\mathcal{U}_0), \tau_{x}^+ >\zeta, \mathcal{E}_\pm \right) = N_0\left(\int_0^\infty ds \int_{(-\infty,0)} \Pi(dy) G_s(y) \right) \\
&= N_0\left(\int_0^\infty ds \int_{(-\infty,0)} \Pi(dy) h(\mathbf{e}(s-)+y,\mathbf{e}(s-)) 1_{\{ \mathbf{e}(s-) \in (0,x) \}} 1_{\{ \mathbf{e}(s-)+y <0 \}} 1_{\{s<\zeta \}} 1_{\{\sup_{r \in (0,s)} \mathbf{e}(r)<x \}}  \right) \\
&= N_0\left(\int_0^\zeta ds 1_{\{ \mathbf{e}(s) \in (0,x) \}} 1_{\{\sup_{r \in (0,s)} \mathbf{e}(r)<x \}} \int_{(-\infty,0)} \Pi(dy) h(\mathbf{e}(s)+y,\mathbf{e}(s))  1_{\{ \mathbf{e}(s)+y <0 \}}  \right) \\
&= N_0\left(\int_0^\zeta ds 1_{\{ \mathbf{e}(s) \in (0,x) \}} 1_{\{s<\tau_0^- \wedge \tau_x^+ \}} \int_{(-\infty,0)} \Pi(dy) h(\mathbf{e}(s)+y,\mathbf{e}(s))  1_{\{ \mathbf{e}(s)+y <0 \}}  \right) \\
&= N_0\left(\int_0^{\tau_x^+ \wedge \tau_0^-} ds 1_{\{ \mathbf{e}(s) >0 \}} \tilde{h}(\mathbf{e}(s))  \right),
\end{align*}
where $\tilde{h}(z)= 1_{\{ z<x\}}\int_{(-\infty,0)} \Pi(dy) h(z+y,z)  1_{\{ z+y <0 \}}$. The last integral can be separated depending if $\tau_x^+>\zeta$ or $\tau_x^+<\zeta$, and for this last case one necessarily has $\tau_x^+ <\tau_0^-$. Therefore,
\begin{align*}
 &N_0\left(\int_0^{\tau_x^+ \wedge \tau_0^-} ds 1_{\{ \mathbf{e}(s) >0 \}} \tilde{h}(\mathbf{e}(s))  \right) \\
 &= N_0\left(\int_0^{\tau_x^+ \wedge \tau_0^-} ds 1_{\{ \mathbf{e}(s) >0 \}} \tilde{h}(\mathbf{e}(s)) ,\tau_x^+ <\zeta  \right) + N_0\left(\int_0^{\tau_x^+ \wedge \tau_0^-} ds 1_{\{ \mathbf{e}(s) >0 \}} \tilde{h}(\mathbf{e}(s))  ,\tau_x^+ >\zeta \right) \\
 &= N_0\left(\int_0^{\tau_x^+} ds 1_{\{ \mathbf{e}(s) >0 \}} \tilde{h}(\mathbf{e}(s)) ,\tau_x^+ <\zeta  \right) + N_0\left(\int_0^{\tau_0^-} ds 1_{\{ \mathbf{e}(s) >0 \}} \tilde{h}(\mathbf{e}(s))  ,\tau_x^+ >\zeta \right) \\
 &= N_0\left(\int_0^{\tau_0^-} ds 1_{\{ \mathbf{e}(s) >0 \}} \tilde{h}(\mathbf{e}(s)) ,\tau_x^+ <\zeta  \right) - N_0\left(\int_{\tau_x^+}^{\tau_0^-} ds 1_{\{ \mathbf{e}(s) >0 \}} \tilde{h}(\mathbf{e}(s)) ,\tau_x^+ <\zeta  \right) \\
 & \quad + N_0\left(\int_0^{\tau_0^-} ds 1_{\{ \mathbf{e}(s) >0 \}} \tilde{h}(\mathbf{e}(s))  ,\tau_x^+ >\zeta \right) \\
 &= N_0\left(\int_0^{\tau_0^-} ds 1_{\{ \mathbf{e}(s) >0 \}} \tilde{h}(\mathbf{e}(s))  \right)- N_0\left(\int_{\tau_x^+}^{\tau_0^-} ds 1_{\{ \mathbf{e}(s) >0 \}} \tilde{h}(\mathbf{e}(s)) ,\tau_x^+ <\zeta  \right).
\end{align*}

For the first term, using equation (20) in \cite{pardo2018excursion} with $\lambda \downarrow 0$ and the fact that $\Phi(0)=0$, we have that 
\begin{align*}
& N_0\left(\int_0^{\tau_0^-} ds 1_{\{ \mathbf{e}(s) >0 \}} \tilde{h}(\mathbf{e}(s))  \right) = \int_0^\infty dz \tilde{h}(z) \\
&= \int_0^x dz \int_{(-\infty,0)} \Pi(dy) h(z+y,z)  1_{\{ z+y <0 \}} \\
&= \int_0^x dz \int_{(-\infty,-z)} \Pi(dy) g(-z-y).
\end{align*}

For the second term, we will use that, from the absence of positive jumps, one necessarily has $\tau_x^+=T_x$. Therefore, from the Markov property at time $\tau_x^+$ to get
\begin{align*}
& N_0\left(\int_{\tau_x^+}^{\tau_0^-} ds 1_{\{ \mathbf{e}(s) >0 \}} \tilde{h}(\mathbf{e}(s)) ,\tau_x^+ <\zeta  \right)= N_0 \left( 1_{\{\tau_x^+ < \tau_0^- \}} \mathbb{E}^0_x \left(\int_0^\zeta ds 1_{\{X_s>0 \}} \tilde{h}(X_s) \right)  \right) \\
&= \mathbb{E}^0_x \left(\int_0^\zeta ds \tilde{h}(X_s) \right) N_0(\tau_x^+ < \tau_0^-) \\
&= U^0(x,\tilde{h}) N_0(\tau_x^+ < \tau_0^-),
\end{align*}
where $U^0$ is the potential of $X$ killed when it goes below $0$ for the first time. From \cite[Corollary 8.8]{kyprianou2014fluctuations}, we know that $$U^0(x,\tilde{h})=\int_0^x (W(x)-W(x-z)) \tilde{h}(z) dz.$$ Finally, from the last identity on the proof of Theorem 3 in \cite{pardo2018excursion} at time $\tau_a^+$, there exists a constant $c_+$ such that
\begin{align*}
N_0(\tau_x^+ < \tau_0^-<\zeta) = c_+ \underline{N}(\tau_x^+ <\zeta, \mathbf{e}(\zeta-)>0),
\end{align*}
whilst from (ii) in the same theorem we have that $$N_0(\tau_x^+ < \tau_0^-=\zeta) = c_+ \underline{N}(\tau_x^+ <\zeta, \mathbf{e}(\zeta-)=0).$$ Therefore, $$N_0(\tau_x^+ < \tau_0^-)=c_+ \underline{N}(\tau_x^+ <\zeta),$$ and it is known that this last quantity equals $\frac{c_+}{W(x)}$ (see \cite[Proposition 15, Ch. VII]{bertoin1996levy}). Putting all the pieces together we obtain
\begin{align*}
\widehat{N}_x \left(g(\mathcal{O}_x), \tau_0^- >\zeta, \mathcal{E}_\pm^x \right) &=  \int_0^x dz \tilde{h}(z) - \frac{c_+}{W(x)} \int_0^x dz (W(x)-W(x-z)) \tilde{h}(z) \\
&= \int_0^x dz \left(1-c_+ + \frac{c_+ W(x-z)}{W(x)} \right) \int_{(-\infty,0)} \Pi(dy) h(z+y,z)1_{\{z+y<0 \}} \\
&= \int_0^x dz \left(1-c_+ + \frac{c_+ W(x-z)}{W(x)} \right) \int_{(-\infty,-z)} \Pi(dy) g(-y-z). 
\end{align*}
With the notation $\widehat{\Pi}(dy)=\Pi(-dy)$ we conclude that
$$\widehat{N}_x \left(g(\mathcal{O}_x), \tau_0^- >\zeta, \mathcal{E}_\pm^x \right)= \int_0^x dz \left(1-c_+ + \frac{c_+ W(x-z)}{W(x)} \right) \int_{(0,\infty)} \widehat{\Pi}(dy+z) g(y).$$ 
\end{proof}

\section{Proofs} \label{prfs}

We start by presenting the proof that the local time process $(L_{\tau(c)}^y, y\in \mathbb{R})$ is infinitely divisible and its Poissonian representation.

\begin{proof}[Proof of Theorem~\ref{PRtauC}]
Recall from Section~\ref{preliminaries} that $$\sigma_s^0=\inf\{t>0: L_t^0(X)>s \}, \quad s>0,$$ is the right continuous inverse of the local time at zero. Then, for those times $s$ such that $\sigma_{s}^0 > \sigma_{s-}^0 := \lim_{t\uparrow s} \sigma_t^0$, we define the excursion away from zero at local time $s$ by $$\mathbf{e}_s^0 (u)=X_{\sigma_{s-}^0+u}, \quad u\in [0,\sigma_s^0 -\sigma_{s-}^0]$$ and the set of excursions $\{(s,\mathbf{e}_s^0):s>0 \}$ is a Poisson point process of intensity $ds\otimes N_0(d\mathbf{e})$. Let $M(ds,d\mathbf{e})$ be the corresponding Poisson point measure. From the occupation formula, for any function $f$ measurable and bounded we have $$\int_\mathbb{R} L_{\tau(c)}^y(X) f(y)dy=\int_0^{\tau(c)} f(X_t)dt.$$ On the other hand, by regularity of $X$ for $(-\infty,0)$ and $(0,\infty)$ (which is a consequence of hypothesis \textbf{(A)}), we can decompose the integral on the right hand side into the contributions of each excursion away from zero as
\begin{align*}
\int_0^{\tau(c)} f(X_t)dt &= \sum_{0<s<c} \int_{\sigma_{s-}^0}^{\sigma_s^0} f(X_t)dt = \sum_{0<s<c} \int_0^{\zeta(\mathbf{e}_s^0)} f(\mathbf{e}_s^0(t))dt\\
&= \sum_{0<s<c} \int_\mathbb{R} L_{\zeta(\mathbf{e}_s^0)}^{y}(\mathbf{e}_s^0)f(y)dy \\
&= \int_\mathbb{R} \left(\sum_{0<s<c} L_{\zeta(\mathbf{e}_s^0)}^{y}(\mathbf{e}_s^0) \right)f(y)dy,
\end{align*}
where in the second in last line we used again the occupation density formula for the local times of each excursion. We can write the last line in terms of $M$ as
$$\int_\mathbb{R} L_{\zeta(\mathbf{e}_s^0)}^{y}(\mathbf{e}_s^0)f(y)dy=\int_\mathbb{R} \left(\int_0^c \int_{D(0,\infty)} L_{\zeta(\mathbf{e})}^y(\mathbf{e}) M(ds,d\mathbf{e}) \right)f(y)dy.$$ Therefore, $$\int_\mathbb{R} L_{\tau(c)}^y f(y)dy = \int_\mathbb{R} \left(\int_0^c \int_{D(0,\infty)} L_{\zeta(\mathbf{e})}^y(\mathbf{e}) M(ds,d\mathbf{e}) \right)f(y)dy,$$ and since this holds for any measurable and bounded function $f$, we conclude that $$L_{\tau(c)}^y(X) = \int_0^c \int_{D(0,\infty)} L_{\zeta(\mathbf{e})}^y(\mathbf{e}) M(ds,d\mathbf{e}),$$ for a.e.-$y$. Then, the full identity follows using the right continuity in the space variable of local times. Finally, if $\widetilde{K}$ is the image of $M$ under the function that maps $\mathbf{e}$ to its local times (hence with intensity $ds\otimes \widetilde{M}(d\ell)$ and $\widetilde{M}(d\ell)$ being the image of $N_0$ under the same map) then $$L_{\tau(c)}^y(X) = \int_0^c \int_{D(0,\infty)} \ell(y) \widetilde{K}(ds,d\ell).$$

We now prove the infinite divisibility. From the strong Markov property of $X$ and the fact that local times are additive, it follows that for any $n\geq 1$ we can split the local times into the information up to $\tau(c/n)$ and the information after as
\begin{align*}
L_{\tau(c)}^y (X)&= L_{\tau(c/n)}^y (X) + L_{\tau(c)}^y(X) \circ \theta_{\tau(c/n)}, \quad y \in \mathbb{R}.
\end{align*}
Since $X_{\tau(c/n)}=0$, we can write $L_{\tau(c)}^y(X) \circ \theta_{\tau(c/n)}$ in terms of the process $(\tilde{X}_t:=X_{\tau(c/n)+t}, t\geq 0)$ which is independent of the information up to $\tau(c/n)$ and has the same law of $X$. Indeed, $$L_{\tau(c)}^y(X) \circ \theta_{\tau(c/n)} = L_{\widetilde{\tau}(\frac{n-1}{n}c)}^y(\widetilde{X}),$$ where $\widetilde{\tau}(\frac{n-1}{n}c)$ is the first time $\widetilde{X}$ accumulates $\frac{n-1}{n}c$ units of local time at $0$. Hence,
$$L_{\tau(c)}^y (X)= L_{\tau(c/n)}^y (X) + L_{\widetilde{\tau}(\frac{n-1}{n}c)}^y(\widetilde{X}).$$ Repeating the argument for $\widetilde{X}$ and using induction, we conclude that we can write
\begin{equation*}
L_{\tau(c)}^{y} (X)\stackrel{(d)}{=} \sum_{k=1}^n L^{y} _{\tau_k(c/n)}(X^{(k)}), \quad y\in \mathbb{R},
\end{equation*}
where $X^{(k)}$ are i.i.d. copies of $X$ starting from zero and $\tau_k(c/n)=\inf\left\{t>0:L_t^0(X^{(k)})>c/n \right\}$ is the first time each one accumulates $c/n$ units of local time at $0$. This proves that the local time process is infinitely divisible. We can identify the associated L\'evy measure $\mu^{(c)}$, that is, a measure that satisfies
\begin{equation*}
\mathbb{E}\left[e^{-\int_\mathbb{R} f(y)L_{\tau(c)}^y(X) dy} \right]=\exp\left\{-\int_{\mathbb{R}_+ ^\mathbb{R}} \left(1-e^{-\int_\mathbb{R} f(y)\omega(y)dy} \right) \mu^{(c)}(d\omega) \right\},
\end{equation*}
for any non-negative, measurable and bounded function $f$. Using the occupation formula one has that $$\int_\mathbb{R} f(y)L_{\tau(c)}^y (X) dy=\int_0^{\tau(c)} f(X_s)ds$$ and decomposing the integral on the right hand side into excursions away from 0 as before, from the exponential formula we obtain
\begin{equation*}
\mathbb{E}\left[e^{-\int_0^{\tau(c)} f(X_s)ds} \right]=\exp\left\{-c\int_{D(0,\infty)} \left(1-e^{-\int_\mathbb{R} f(y)L_\zeta^y (\mathbf{e})dy} \right) N_0(d\mathbf{e}) \right\}.
\end{equation*}
The latter yields the following identity for the L\'evy measure of local times up to $\tau(c)$,
\begin{equation*}
\mu^{(c)}(d\omega)=cN_0(L_\zeta^\cdot \in d\omega).
\end{equation*}
\end{proof}

We now present the proof of the refined version of the L\'evy measure $\mu^{(c)}$.

\begin{proof}[Proof of Theorem~\ref{LevyMc}]
For the representation of $N_0$ in the set $\mathcal{E}_\pm$ we just condition on the overshoot and undershoot $(\mathcal{O}_0,\mathcal{U}_0)$ of the excursion $\mathbf{e}$ and use the strong Markov property at time $\tau_0^-$ to get the conditional independence of the paths $\underleftarrow{\mathbf{e}}$ and $\underrightarrow{\mathbf{e}}$. Lemma~\ref{ouN0} provides an expression for the joint density of $(\mathcal{O}_0,\mathcal{U}_0)$, proving the second expression. And for the last one, using the occupation formula and decomposing the path in excursions away from the supremum as in the proof of Proposition~\ref{intBarN}, one obtains for $u<0$
\begin{align*}
\mathbb{E}^0_u\left(e^{-\int_{(-\infty,0)} f_-(y)L_\zeta^y dy} \right)&=\mathbb{E}_u \left(e^{-\int_{0}^{\tau_0^+} f_-(X_r)dr} \right) \\
&= \exp\left\{-\int_{u}^0 ds \overline{N} \left[1-e^{-\int_0^\zeta f_-(s-\mathbf{e}(r)) dr} \right]  \right\} \\
&= \mathcal{W}_{f_{-,u}}(-u).
\end{align*}
Proceeding analogously for $b>0$, we obtain the desired expression for the Laplace transform. 
\end{proof}

The next proof concerns the Poissonian representation of $(L_{\tau(c)}^y, y\in \mathbb{R})$ involving the overshoots and undershoots.

\begin{proof}[Proof of Theorem~\ref{PRov}]
We only show equation \eqref{PRovPos}, as the proof of \eqref{PRovNeg} is analogous. We will do it by considering Laplace transforms. Let $f:(0,\infty)\to \mathbb{R}_+$ measurable and bounded and call $G(y)$ the right hand side of \eqref{PRovPos}. For an element $Y \in \mathcal{M}$, the space of Poisson random measures, we will denote by $\langle Y,F \rangle$ at the integral of a functional $F$ with respect to $Y$. In particular, when $F$ is the functional $\int_0^\infty f(y) L_\zeta ^{y-\cdot} (\cdot)1_{\{y-\cdot >0\}}dy$, recalling that a Poisson random measure $Y \in \mathcal{M}$ is written in terms of its atoms $Y=\sum_{r>0} \delta_{(r,\underline{\mathbf{e}}_r)}$, we have that $$\left\langle Y,F \right\rangle= \left\langle Y, \int_0^\infty f(y)L_\zeta ^{y-\cdot}(\cdot)1_{\{y-\cdot>0\}}dy \right\rangle =  \sum_{r>0} \int_0^\infty f(y) L_{\zeta}^{y-r}(\underline{\mathbf{e}}_r)1_{\{y-r>0\}} dy.$$ 

Using that $M_+$ and $M_\pm^{1}$ are independent we have that
\begin{align*}
&\mathbb{E} \left[\exp \left\{-\int_0^\infty f(y)G(y)dy \right\} \right] \\
&= \mathbb{E} \left[\exp \left\{-\int_0^c \int_{D(0,\infty)} \left( \int_0^\infty f(y)L^y_\zeta (\mathbf{e})dy \right) M_+(ds,d\mathbf{e}) \right\} \right] \\
& \times \mathbb{E} \left[\exp \left\{-\int_0^c \int_{(0,\infty)} \int_{\mathcal{M}} \left\langle Y, \left( \int_0^\infty f(y)L^{y-\cdot}_\zeta (\cdot)1_{\{y-\cdot >0\}} dy \right) \right\rangle M_\pm^1 (ds,db,dY) \right\} \right].
\end{align*}
Using the exponential formula, the first term equals
\begin{align*}
& \mathbb{E} \left[\exp \left\{-\int_0^c \int_{D(0,\infty)} \left( \int_0^\infty f(y)L^y_\zeta (\mathbf{e})dy \right) M_+(ds,d\mathbf{e}) \right\} \right] \\
&= \exp \left\{-\int_0^c \int_{D(0,\infty)} ds N_0(d\mathbf{e},\mathcal{E}_+) (1-e^{-\int_0^\infty f(y)L_\zeta^y dy}) \right\} \\
&= \exp \left\{-c N_0 \left(1-e^{-\int_0^\infty f(y)L_\zeta^y dy},\mathcal{E}_+ \right) \right\}.
\end{align*}
Similarly,
\begin{align*}
&\mathbb{E} \left[\exp \left\{-\int_0^c \int_{(0,\infty)} \int_{\mathcal{M}} \left\langle Y, \left( \int_0^\infty f(y)L^{y-\cdot}_\zeta (\cdot)1_{\{y-\cdot>0\}} dy \right) \right\rangle M_\pm^1 (ds,db,dY) \right\} \right] \\
&= \exp \left\{-\int_0^c  \int_{(0,\infty)} \int_{\mathcal{M}} ds N_0(\mathcal{O}_0 \in db, \mathcal{E}_\pm) \widehat{\kappa}(b,dY) \left(1-e^{-\langle< Y,\int_0^\infty f(y) L_\zeta^{y-\cdot} (\cdot)1_{\{y-\cdot>0\}} dy \rangle}  \right)  \right\} \\
&= \exp \left\{- c  \int_{(0,\infty)} db \Pi(-\infty,-b) \mathbb{E} \left(1-\exp \left\{-\sum_{0<r<b} \int_0^\infty f(y) L_\zeta^{y-r} (\underline{\mathbf{e}}_r)1_{\{y-r>0\}} dy  \right\} \right)   \right\} \\
&= \exp \left\{-c \int_{(0,\infty)} db \Pi(-\infty,-b) \left( 1-\exp \left\{-\int_0^{b\wedge y} dr \overline{N} \left(1-e^{-\int_0^\infty f(y) L_\zeta^{y-r} dy} \right)  \right\} \right) \right\}.
\end{align*}
Therefore, comparing with the Laplace transform in Theorem~\ref{LevyMc}, we conclude that $$\mathbb{E} \left[\exp \left\{-\int_0^\infty f(y)G(y)dy \right\} \right] = \mathbb{E} \left[\exp \left\{-\int_0^\infty f(y)L_{\tau(c)}^y (X) dy \right\} \right],$$ which proves the Poissonian representation of local times.
\end{proof}

The following proof concerns the Poissonian representation and the infinite divisibility property of the local time process $(L_{\tau_a^+}^{a-z}(X), z \geq 0)$.

\begin{proof}[Proofs of Theorem~\ref{PRtauA}]

Let $M'(ds, d\mathbf{e})$ be the Poisson point measure of excursions away from the supremum and $f:\mathbb{R}^2 \to \mathbb{R}_+$ measurable and bounded. Again, hypothesis \textbf{(A)} implies that 0 is regular for both $(-\infty,0)$ and $(0,\infty)$ and \cite[Th.6.7]{kyprianou2014fluctuations} implies that the set of times in which $X$ reaches new suprema on compact intervals has Lebesgue measure 0. Therefore, we have the following decomposition: 
\begin{eqnarray*}
\int_0^{\tau_a^+} f(S_t-X_t,X_t)dt&=& \sum_{0<v<a} \int_{\tau_{v-}^+}^{\tau_v^+} f(S_t-X_t,X_t) dt = \sum_{0<v<a} \int_0^{\zeta(\mathbf{e}_v)} f(\mathbf{e}_v(t),v-\mathbf{e}_v(t)) dt.
\end{eqnarray*}

Applying the occupation density formula, we can express the latter quantity in terms of $M'$ as 
\begin{eqnarray*}
\sum_{0<v<a} \int_0^{\zeta(\mathbf{e}_v)} f(\mathbf{e}_v(t),v-\mathbf{e}_v(t)) dt &=& \sum_{0<v<a} \int_\mathbb{R} f(u,v-u)L_{\zeta_v(\mathbf{e})}^u(\mathbf{e}) du \\
&=& \sum_{0<v<a} \int_{[0,\infty)} f(u,v-u)L_{\zeta_v(\mathbf{e})}^u(\mathbf{e}) du \\
&=& \int_{[0,\infty)} \left(\int_0^a \int_{D(0,\infty)} f(u,s-u)L_{\zeta(\mathbf{e})}^u(\mathbf{e}) M'(ds,d\mathbf{e}) \right) du.
\end{eqnarray*}

Now we use Fubini's theorem and the change of variables $z=s-u$ to obtain
\begin{eqnarray*}
\int_0^{\tau_a^+} f(S_t-X_t,X_t)dt&=&\int_0^a \int_{D(0,\infty)} \left(\int_\mathbb{R} f(s-z,z)L_{\zeta(\mathbf{e})}^{s-z}(\mathbf{e}) 1_{\{s-z\geq 0 \}} dz \right) M'(ds,d\mathbf{e}) \\
&=& \int_{(-\infty,a]} \left(\int_0^a \int_{D(0,\infty)} L_{\zeta(\mathbf{e})}^{s-z}(\mathbf{e})1_{\{s-z \geq 0 \}} M'(ds,d\mathbf{e})\right)f(s-z,z) dz.
\end{eqnarray*}

On the other hand, from the occupation formula for $X$ and taking $f$ only depending on the second entry we have
\begin{eqnarray*}
\int_{(-\infty,a]} f(z)L_{\tau_a^+}^z(X)dz= \int_0^{\tau_a^+} f(X_t)dt =\int_{(-\infty,a]} \left(\int_0^a \int_{D(0,\infty)} L_{\zeta(\mathbf{e})}^{s-z}(\mathbf{e})1_{\{s-z > 0 \}} M'(ds,d\mathbf{e}) \right)f(z) dz,
\end{eqnarray*}
and since this holds for any test function $f$, we conclude that $$L_{\tau_a^+}^z(X) = \int_0^a \int_{D(0,\infty)} L_{\zeta(\mathbf{e})}^{s-z}(\mathbf{e})1_{\{s-z > 0 \}} M'(ds,d\mathbf{e})$$ for almost every $z\in (-\infty,a]$. The last equality holds for all $z$ by using the right continuity of local times.

To complete the proof we focus on the reversed process $L_{\tau_a^+}^{a-z}(X)$, with $z\geq 0$. Since Lebesgue measure is invariant under translations, the Poisson point measure $M(ds,d\mathbf{e}):=M'(a-ds,d\mathbf{e})$ for $s\in [0,a]$ has the same intensity as $M$ restricted to $[0,a]$. Hence, we can write
\begin{eqnarray*}
L_{\tau_a^+}^{a-z}(X)&=& \int_0^a \int_{D(0,\infty)} L_{\zeta(\mathbf{e})}^{s-a+z}(\mathbf{e}) 1_{\{s-a+z > 0 \}} M'(ds,d\mathbf{e}) \\
&=& \int_0^a \int_{D(0,\infty)} L_{\zeta(\mathbf{e})}^{z-s}(\mathbf{e}) 1_{\{z-s > 0 \}} M(ds,d\mathbf{e}) \\
&=& \begin{cases}
\int_0^z \int_{D(0,\infty)} L_{\zeta(\mathbf{e})}^{z-s}(\mathbf{e}) M(ds,d\mathbf{e}), & z\in [0,a] \\
\int_0^a \int_{D(0,\infty)} L_{\zeta(\mathbf{e})}^{z-s}(\mathbf{e}) M(ds,d\mathbf{e}), & z \geq a.
\end{cases} \\
&=& \begin{cases}
\int_0^z \int_{D(0,\infty)} \mathbf{e}(z-s) K(ds,d\mathbf{e}), & z\in [0,a] \\
\int_0^a \int_{D(0,\infty)} \mathbf{e}(z-s) K(ds,d\mathbf{e}), & z \geq a.
\end{cases},
\end{eqnarray*}
which concludes the result.

We now show that local times are infinitely divisible. We split the information up to $\tau_{a/n}^+$ and after to obtain for any $n \geq 1$ that
\begin{align*}
L_{\tau_a^+}^x (X) &= L_{\tau_{a/n}^+}^x (X) + L_{\tau_a^+}^x (X) \circ \theta_{\tau_{a/n}^+}, x\in \mathbb{R}.
\end{align*}
Since $X_{\tau_{a/n}^+}=a/n$, we can write $L_{\tau_a^+}^x (X) \circ \theta_{\tau_{a/n}^+}$ in terms of the process $(\widetilde{X}_t:= X_{\tau_{a/n}+t}-a/n, t\geq 0)$, which is independent of the information up to $\tau_{a/n}^+$ and has the same law as $X$ started from $0$. Indeed, $$L_{\tau_a^+}^x (X) \circ \theta_{\tau_{a/n}^+}=L_{\widetilde{\tau}_{\frac{n-1}{n}a}^+}^{x-a/n}(\widetilde{X}), \quad x\in \mathbb{R},$$ where $\widetilde{\tau}_{\frac{n-1}{n}a}^+$ is the first passage time above $\frac{n-1}{n}a$ for $\widetilde{X}$. Applying this argument inductively, we conclude that $$L_{\tau_a^+}^x (X) \stackrel{\text{(d)}}{=} \sum_{k=1}^n L_{\tau_{a/n}^{+,k}}^{x-\frac{k-1}{n}a}(X^{(k)}), \quad x\in \mathbb{R},$$ where $X^{(k)}$ are i.i.d. copies of $X$ started from 0 and $\tau_{a/n}^{+,k}$ are the corresponding first passage times above $a/n$. This implies the infinite divisibility of $(L_{\tau_a^+}^x (X),x\in \mathbb{R})$. 

We now identify the corresponding L\'evy measure. Indeed, using that $$\int_0^\infty f(y)L_{\tau_a^+}^{a-y}(X) dy=\int_{-\infty}^a f(a-z)L_{\tau_a^+}^{z}(X) dz=\int_0^{\tau_a^+} f(a-X_s)ds$$ and decomposing the last integral into excursions away from the supremum, again from the exponential formula one gets
\begin{align*}
\mathbb{E}\left[e^{-\int_0^{\tau_a^+} f(a-X_s)ds} \right]&=\exp\left\{-\int_0^a \int_{D(0,\infty)} \left(1-e^{-\int_0^\zeta f(a-s+\mathbf{e}(r))dr }  \right) \overline{N}(d\mathbf{e}) ds \right\} \\
&= \exp\left\{-\int_0^a \int_{D(0,\infty)} \left(1-e^{-\int_0^\infty f(a-s+z)L_\zeta^{z} (\mathbf{e})dz} \right) \overline{N}(d\mathbf{e}) ds \right\}\\
&= \exp\left\{-\int_0^a \int_{D(0,\infty)} \left(1-e^{-\int_0^\infty f(s+z)L_\zeta^{z} (\mathbf{e})dz} \right) \overline{N}(d\mathbf{e}) ds \right\} \\
&= \exp\left\{-\int_0^a \int_{D(0,\infty)} \left(1-e^{-\int_0^\infty f(y)L_\zeta^{y-s} (\mathbf{e}) 1_{\{y-s>0 \}}dy} \right) \overline{N}(d\mathbf{e}) ds \right\}.
\end{align*}

Therefore, we can write
\begin{eqnarray*}
\mathbb{E}\left[e^{-\int_0^\infty f(y)L_{\tau_a^+}^{a-y}(X) dy} \right]=\exp\left\{-\int_{\mathbb{R}_+ ^\mathbb{R}} \left(1-e^{-\int_\mathbb{R} f(y)\omega(y)dy} \right) \nu^{(a)}(d\omega) \right\},
\end{eqnarray*}
where
\begin{equation*}
\nu^{(a)}(d\omega)=\int_0^a ds \overline{N}(L_\zeta ^{\cdot-s}1_{\{\cdot-s >0 \}} \in d\omega).
\end{equation*}
\end{proof}
The following two proofs show the law of local times prior to the first hitting time of a positive level $x$ and between the first and last visits to it.

\begin{proof}[Proof of Theorem~\ref{PR1Tx}]

For $x>0$, it might happen that $\tau_x^+ = T_x$, which implies that the excursion reaches $x$ for the first time continuously and coming from below. In this case, $L_{T_x}^y=0$ for any $y>x$. Then, we only have to deal with the case $\tau_x^+ < T_x <\zeta$ in which the excursion goes above $x$ for the first time by a jump, hence having a first strictly positive overshoot $\mathcal{O}_x$ with respect to $x$. Conditionally on $\mathcal{O}_x=b>0$, by the strong Markov property the path between $\tau_x^+$ and $T_x$ has the law $\widehat{\mathbb{E}}_{x+b}^{x}$, which is the law of $\widehat{X}$ started from $x+b$ and killed when it reaches $x$ for the first time. We have seen in Section~\ref{srkt} that the measure $\widehat{\mathbb{E}}_b^0$ that can be further decomposed into excursions away from the infimum via the kernel $\widehat{\kappa}$, and by space invariance we can do the same for $\widehat{\mathbb{E}}_{x+b}^{x}$. This, together with the expression for the law of the overshoots from Lemma~\ref{ovN1} implies the first identity.

For the particular case we have from the occupation formula that $$\int_x^\infty f(y)L_{T_x}^y(\mathbf{e}) dy = \int_0^{T_x} f(\mathbf{e}(s))1_{\{\mathbf{e}(s)>x \}} ds,$$ and applying the Markov property at time $\tau_x^+< T_x$ we obtain
\begin{align*}
& \overline{N} \left[ \exp\left\{-\int_x^\infty f(y)L_{T_x}^y dy \right\} , \tau_x^+ <T_x<\zeta \right]= \overline{N} \left[ \exp\left\{-\int_0^{T_x} f(\mathbf{e}(s))1_{\{\mathbf{e}(s)>x \}} ds \right\},\tau_x^+ <T_x<\zeta \right] \\
&\quad = \overline{N} \left[1_{\{ \tau_x^+ <T_x<\zeta \}} \widehat{\mathbb{E}}^x_{\mathbf{e}(\tau_x^+)} \left(\exp\left\{-\int_0^\zeta f(X_s)1_{\{ X_s>x \}}ds \right\} \right)\right] \\
& \quad = \int_{(0,\infty)} \overline{N}(\mathcal{O}_x \in db , \tau_x^+<T_x<\zeta) \widehat{\mathbb{E}}^x_{x+b} \left(\exp\left\{-\int_0^\zeta f(X_s)1_{\{ X_s>x \}}ds \right\} \right),  
\end{align*}
where $\widehat{\mathbb{E}}^x_{x+b}$ is the law of $\widehat{X}$ started from $x+b$ and killed at $\tau_x^-$. For the expected value, we use duality to obtain $$\widehat{\mathbb{E}}^x_{x+b} \left(\exp\left\{-\int_0^\zeta f(X_s)ds \right\} \right)=\mathbb{E}_{0} \left(\exp\left\{-\int_0^{\tau_b^+} f(x+b-X_s) ds \right\} \right)$$ and we decompose as usual the right hand side into excursions away from the supremum: 
\begin{align*}
\mathbb{E}_{0} \left(\exp\left\{-\int_0^{\tau_b^+} f(x+b-X_s) ds \right\} \right)&= \exp \left\{-\int_0^b ds \overline{N}\left[1-e^{-\int_0^\zeta dr f(x+b-s+\mathbf{e}(r))} \right] \right\}.
\end{align*}
On one hand, by defining $f_{x,b}(z):=f(x+b-z), z<b$ we have that $$\exp \left\{-\int_0^b ds \overline{N}\left[1-e^{-\int_0^\zeta dr f(x+b-s+\mathbf{e}(r))} \right] \right\}=\exp \left\{-\int_0^b ds \overline{N}\left[1-e^{-\int_0^\zeta dr f_{x,b}(s-\mathbf{e}(r))} \right] \right\}= \mathcal{W}_{f_{x,b}}(b).$$ On the other hand, by making the change of variables $u=b-s$ we get
\begin{align*}
\exp \left\{-\int_0^b ds \overline{N}\left[1-e^{-\int_0^\zeta dr f(x+b-s+\mathbf{e}(r))} \right] \right\}=\exp \left\{-\int_0^b du \overline{N}\left[1-e^{-\int_0^\zeta dr f(x+u+\mathbf{e}(r))} \right] \right\}.
\end{align*}
Finally, from the occupation formula we have that
\begin{align*}
\int_0^\zeta f(x+s+\mathbf{e}(r))dr&=\int_0^\infty f(x+s +y)L_{\zeta}^y dy = \int_{x+s}^\infty f(y) L_{\zeta}^{y-x-s} dy = \int_{x}^\infty f(y) L_{\zeta}^{y-x-s} 1_{\{y>x+s \}} \\
&= \int_{x}^\infty f(y) L_{\zeta}^{y-x-s} 1_{\{y-s>x \}},
\end{align*}
which, when put together with the previous expressions, gives the desired result.
\end{proof}

\begin{proof}[Proof of Theorem~\ref{PR1zeta}]
From the strong Markov property at time $T_x$ we have that
\begin{align*}
& \overline{N} \left( \exp\left\{-\int_x^\infty f(y) L_{\zeta}^y \circ \theta_{T_x} dy  \right\},H>x \right) = \overline{N} \left( 1_{\{H>x \}}\widehat{\mathbb{E}}_x^0 \left[\exp\left\{-\int_x^\infty f(y) L_{\zeta}^y dy \right\} \right] \right) \\
&= \overline{N}(H>x) \widehat{\mathbb{E}}_x^0 \left[\exp\left\{-\int_x^\infty f(y) L_{\zeta}^y dy  \right\} \right]. 
\end{align*}

We will now prove via Laplace transforms that the local time at $x$, $L_{\zeta}^x$, follows an exponential distribution of parameter $q_x^0=\widehat{N}_x(\tau_0^-<\zeta)$ under $\widehat{\mathbb{E}}_x^0$. Under this measure, the lifetime $\zeta$ coincides with the first passage time below zero of a special excursion away from $x$. This special excursion away from $x$ is the first one that satisfies that $T_0$ occurs before its own lifetime. Conditioning on the left endpoint of this special excursion, we have that 
\begin{equation*}
\widehat{\mathbb{E}}_{x}^0 \left(e^{-\lambda L_\zeta^x} \right)=\widehat{\mathbb{E}}_{x}^0 \left( \sum_{s\in \mathcal{G}} 1_{\{T_{0} \circ \theta_s <\zeta \}}e^{-\lambda L_s^x} \right),
\end{equation*}
where $\mathcal{G}$ is precisely the set of left endpoints of the intervals in which excursions away from $x$ occur and the event $\{T_{0}\circ \theta_s< \zeta \}$ represents that the last excursion reaches $0$ and hence it does not return to $x$. Using the compensation formula for the excursions away from $x$, we have
\begin{eqnarray*}
\widehat{\mathbb{E}}_{x}^0 \left( \sum_{s\in \mathcal{G}} 1_{\{T_y \circ \theta_s <\zeta \}}e^{-\lambda L_s^x} \right)&=&\widehat{\mathbb{E}}_{x}^0 \left( \int_0^\infty dL_s^{x} e^{-\lambda L_s^{x}} \widehat{N}_{x}(T_0 < \zeta) \right) \\
 &=& \widehat{N}_{x}(T_0 < \zeta) \widehat{\mathbb{E}}_{x}^0 \left( \int_0^\infty dL_s^{x} e^{-\lambda L_s^{x}}  \right).
\end{eqnarray*}
Now, with the change of variables $u=L_s^{x}$ (with inverse $\sigma_u^{x}=\inf\{s>0:L_s^{x}>u \}$) we have
\begin{align*}
\widehat{\mathbb{E}}_{x}^0 \left( \int_0^\infty dL_s^{x} e^{-\lambda L_s^{x}}  \right)&=\widehat{\mathbb{E}}_{x}^0 \left( \int_0^{L_{\zeta}^{x}} du e^{-\lambda u}  \right) =\widehat{\mathbb{E}}_{x}^0 \left( \int_0^{\infty} du 1_{\{u<L_\zeta^{x} \}} e^{-\lambda u}  \right) \\
&= \int_0^{\infty} du e^{-\lambda u} \mathbb{E}_{x}^0 \left(  1_{\{u<L_\zeta^{x} \}}  \right).
\end{align*}
For the last expectation, we now use the exponential formula for the excursions away from $x$. Denote by $\{(v,\mathbf{e}^{x}_v):v>0 \}$ the set of excursions away from $x$. The event $\{u<L_\zeta ^{y_1} \}$ is equivalent to $\{\sigma_u^{x} <\zeta \}$, which can also be written as $\{ \# \{(v,\mathbf{e}_v^{x}): 0<v<u, T_{0}(\mathbf{e}_v^{x}) < \zeta(\mathbf{e}_v^{x}) \}=0 \}$, that is, none of the excursions up to local time $u$ has reached $0$; otherwise, since we are working under $\widehat{\mathbb{E}}_{x}^0$, there would be an excursion with infinite length causing the inverse of the local time to explode. Therefore, $$\mathbb{E}_{x}^0 \left(  1_{\{u<L_\zeta^{y_1} \}} \right)= \exp \left\{-u \widehat{N}_{x}(T_0<\zeta) \right\} $$ and hence, 
\begin{align*}
\widehat{\mathbb{E}}_{x}^0 \left( \int_0^\infty dL_s^{y_1} e^{-\lambda L_s^{x}}  \right)&= \int_0^{\infty} du e^{-\lambda u} e^{-u \widehat{N}_{x}(T_0<\zeta)} = \frac{1}{\lambda + \widehat{N}_{x}(T_0<\zeta)}.
\end{align*}
Going back to the original computation we conclude that $$\mathbb{E}_{x}^0 \left(e^{-\lambda L_\zeta^{x}} \right)=\frac{\widehat{N}_{x}(T_0<\zeta)}{\lambda+ \widehat{N}_{x}(T_0<\zeta)}.$$ In particular, $L_\zeta^{x}$ follows an exponential distribution of parameter $\widehat{N}_{x}(T_0<\zeta)$ under $\mathbb{E}_{x}^0$. The claim follows by noting that $T_0=\tau_0^-$ because of the assumption of unbounded variation.

Now, conditionally on $L_\zeta^x$, we decompose the functional $\int_x^\infty f(y) L_{\zeta}^y dy$ into the sum of the contribution of each excursion $(\mathbf{e}^x_s,0<s<L_\zeta^x)$ away from $x$. Hence, by the exponential formula
\begin{align*}
\widehat{\mathbb{E}}_x^0 \left[\exp\left\{-\int_x^\infty f(y) L_{\zeta}^y dy  \right\} \right]&= \int_0^\infty dr q_x^0 e^{-r q_x^0} \widehat{\mathbb{E}}_x^0 \left[ \exp\left\{ -\sum_{0<s<r} \int_x^\infty f(y)L_{\zeta}^y(\mathbf{e}^x_s) dy \right\} \right] \\
&= \int_0^\infty dr q_x^0 e^{-r q_x^0} \exp\left\{-r \widehat{N}_x \left[1-e^{-\int_x^\infty f(y) L_\zeta^y dy} , \tau_0^->\zeta \right] \right\} \\
&= \widehat{\mathbb{E}}^x_0 \left[\exp\left\{-L_\zeta^x \widehat{N}_x \left[1-e^{-\int_x^\infty f(y) L_\zeta^y dy} , \tau_0^- >\zeta  \right] \right\} \right],
\end{align*}
which implies the first part of the theorem.

For the second part, we know that under $\widehat{N}_x$, excursions away from $x$ are partitioned into those completely above $x$ ($\mathcal{E}^x_+$), those completely below $x$ ($\mathcal{E}^x_-$) and those starting below and then jumping above $x$ ($\mathcal{E}_\pm^x$). Excursions in $\mathcal{E}^x_-$ do not contribute to the local times of levels bigger than $x$ and hence we restrict to the other two sets. In $\mathcal{E}^x_+$, the condition $\tau_0^- >\zeta$ is automatically fulfilled, so we can omit it. Finally, for an excursion in $\mathcal{E}_\pm^x$, there is an unique positive overshoot $\mathcal{O}_x$ at time $\tau_x^+$ and the part of the excursion contributing to levels $y>x$ is from $\tau_x^+$ to $\zeta$, which, conditionally on $\mathcal{O}_x=b$, we have seen that has the law $\widehat{\mathbb{E}}_{x+b}^x$ and therefore can be decomposed into excursions away from the infimum, as in the proof of Theorem~\ref{PR1Tx}. Using the computations there we have that
\begin{align*}
& \widehat{N}_x \left[1-e^{-\int_x^\infty f(y) L_\zeta^y dy} , \tau_0^- >\zeta, \mathcal{E}_\pm^x  \right] \\
&= \int_0^\infty \widehat{N}_x(\mathcal{O}_x \in db,\tau_0^- >\zeta, \mathcal{E}_\pm ^x) \widehat{\mathbb{E}}_{x+b}^x \left[1- e^{-\int_x^\infty f(y) L_\zeta^y dy} \right] \\
&= \int_0^\infty \widehat{N}_x(\mathcal{O}_x \in db,\tau_0^- >\zeta, \mathcal{E}_\pm ^x)\left[1- \exp \left\{-\int_0^b du \overline{N}\left[1-e^{-\int_{x}^\infty f(y) L_{\zeta}^{y-x-s} 1_{\{y-s>x \}}} \right] \right\} \right],
\end{align*}
which concludes the proof.
\end{proof}

Finally, the proofs of the proposition relative to the scale function $\mathcal{W}_f$ and the propositions from Section~\ref{auxR} can be found below.

\begin{proof}[Proof of Proposition~\ref{intBarN}]
The function $\mathcal{W}_f$ from \eqref{newScaleF} is well defined, since
\begin{align*}
\lim_{b\to -\infty} \frac{W_f(0,b)}{W_f(x,b)}&= \lim_{b\to -\infty} \frac{W(-b)}{W(x-b)} \exp\left\{- \int_0^x ds \overline{N}\left(1-e^{-\int_0^\zeta dr f(s-\mathbf{e}(r))} ,H<s-b\right) \right\} \\
&= \lim_{b\to \infty} \frac{W(b)}{W(x+b)} \exp\left\{- \int_0^x ds \overline{N}\left(1-e^{-\int_0^\zeta dr f(s-\mathbf{e}(r))} ,H<s+b\right) \right\}.
\end{align*}
The first quotient equals $\mathbb{P}_0(\tau_x^+ < \tau_{-b}^-)$ and converges to $\mathbb{P}_0 (\tau_x^+< \infty)$ when $b\to \infty$, which is equal to one because of hypothesis \textbf{(B1)} or \textbf{(B2)}. And when $b \to \infty$, the condition $H<s+b$ inside the exponential is replaced by $H<\infty$ but again, since under \textbf{(B1)} or \textbf{(B2)} we have $\limsup_{t\to \infty} X_t = \infty$, this implies that the excursions with infinite height have zero mass under $\overline{N}$.

Now let us check that $G_f$ defined as in the statement coincides with $\mathcal{W}_f$ and solves equation \eqref{intBarN1}. For $x>0$, assumptions \textbf{(B1)} and \textbf{(B2)} imply that $\limsup_{t \to \infty} X_t =\infty$ and therefore $\tau_0^- <\infty$ $\widehat{\mathbb{P}}_x -$a.s. Then, for $x>0$,
\begin{align*}
G_f(x)&= \widehat{\mathbb{E}}_x \left[\exp \left\{-\int_0^{\tau_0^-} ds f(x-X_s) \right\} \right] \\
&= 1- \widehat{\mathbb{E}}_x \left[1-\exp\left\{-\int_0^{\tau_0^-} ds f(x-X_s) \right\} \right] \\
&= 1- \widehat{\mathbb{E}}_x \left[ \int_0^{\tau_0^-} dt f(x-X_t)\exp\left\{-\int_t^{\tau_0^-} ds f(x-X_s) \right\} \right],
\end{align*}
where in the last line we use the fact that if $\widetilde{f}(t):= \exp\left\{-\int_t^{\tau_0^-} ds f(x-X_s) \right\}$, $t\in [0,\tau_0^-]$, then $\widetilde{f}'(t)=f(x-X_t) \widetilde{f}(t)$ and $\widetilde{f}(\tau_0^-)-\widetilde{f}(0)=1-\exp\left\{-\int_0^{\tau_0^-} ds f(x-X_s) \right\}$. Applying the Markov property at time $t$ inside the expression in the last line we obtain 
\begin{align*}
G_f(x)&= 1- \widehat{\mathbb{E}}_x \left[ \int_0^{\tau_0^-} dt f(x-X_t) \widehat{\mathbb{E}}_{X_t} \left[\exp\left\{-\int_0^{\tau_0^-} ds f(x-X_s) \right\} \right] \right] \\
&=1- \widehat{\mathbb{E}}_x \left[ \int_0^{\tau_0^-} dt f(x-X_t) G_f(X_t) \right],
\end{align*}
which can be written in terms of the potential $\widehat{U}^0$ of $\widehat{X}$ killed at $\tau_0^-$ as $$G_f(x)=1-\int_0^\infty \widehat{U}^0(x,dz)f(x-z)G_f(z).$$ According to \cite[Corollary 8.8]{kyprianou2014fluctuations}, $\widehat{U}^0$ can also be expressed in terms of scale functions as $$G_f(x)=1-\int_0^\infty dz(W(x)-W(x-z))f(x-z)G_f(z).$$ This and the fact that $G_f(0)=0$ because of the assumption of unbounded variation, implies the equation \eqref{intBarN1}. For the other part, we notice that
\begin{align*}
G_f(x)&= \widehat{\mathbb{E}}_x \left[\exp \left\{-\int_0^{\tau_0^-} ds f(x-X_s) \right\} \right] = \mathbb{E}_{-x} \left[\exp \left\{-\int_0^{\tau_0^+} ds f(x+X_s) \right\} \right] \\
&= \mathbb{E}_{0} \left[\exp \left\{-\int_0^{\tau_x^+} ds f(X_s) \right\} \right],
\end{align*}
and decomposing into excursions away from the supremum between 0 and $x$ and using the exponential formula we get
\begin{align*}
G_f(x)&=\exp\left\{-\int_0^x ds \overline{N} \left[1-e^{-\int_0^\zeta du f(s-\mathbf{e}(u))} \right] \right\} \\
&= \exp\left\{-\int_0^x ds g_f(s) \right\},
\end{align*}
which also proves that $G_f(x)=\mathcal{W}_f(x)$. The relation $\frac{d}{dx}(- \log G_f)(x)=g_f(x)$ now follows from a simple differentiation of the latter equation.
\end{proof}

\begin{proof}[Proof of Proposition \ref{Prop1}]

We start by noticing that $$\overline{N} \left(\left(1-e^{-\lambda L_\zeta ^x- \int_x^\infty f(y) L_\zeta^y dy} \right)1_{\{H<x\}}\right)=0,$$ since for an excursion to accumulate local time at $x$ and the levels above, it must have height bigger that $x$. Then, on the event $H>x$, applying  the Markov property at time $T_x$ we have that

\begin{eqnarray*}
\overline{N} \left(1-e^{-\lambda L_\zeta^x -\int_x^\infty f(y) L_\zeta^y dy} \right)&=&\overline{N} \left(\left(1-e^{-\lambda L_\zeta ^x- \int_x^\infty f(y) L_\zeta^y dy} \right)1_{\{H>x\}}\right)\\
&=& \overline{N} \left(\left(1-e^{-\int_x^\infty f(y) L_{T_x}^y dy} \right)1_{\{H>x\}}\right) \\
&&\qquad + \overline{N} \left(e^{- \int_x^\infty f(y) L_{T_x}^y dy} \hat{\mathbb{E}}_x \left[1-e^{-\lambda L^x_{\tau_0^-} - \int_x^\infty f(y) L_{\tau_0^-}^y dy} \right]1_{\{H>x\}} \right).
\end{eqnarray*}

It is known that under $\widehat{\mathbb{E}}_x$ (see the argument on the proof of Theorem~\ref{PR1zeta}), the total local time at $x$ up to   the first passage time below $0,$ $L_{\tau_0^-}^x,$ follows an exponential distribution of parameter $q_x=\widehat{N}_x(\tau_0^-<\zeta)$. This is due to the fact that $L_{\tau_0^-}^x$ is the first time there appears an excursion from $x$ at which the dual process passes below zero before ending. Also note that the total local time at a levels $y_1,\dots y_n >x,$ is the addition of the cumulated local time inside each excursion from $x.$ This implies
\begin{eqnarray*}
\widehat{\mathbb{E}}_x \left[1-e^{-\lambda L^x_{\tau_0^-} - \int_x^\infty f(y) L_{\tau_0^-}^y dy} \right]=\widehat{\mathbb{E}}_x \left[1-e^{-\lambda L^x_{\tau_0^-} - \sum_{0<t<L_{\tau_0^-}^x} \int_x^\infty f(y) L_\zeta^y (\mathbf{e}_t^x) dy} \right].
\end{eqnarray*}

Denote by $(\mathbf{e}_s^x (\tau_0^->\zeta_s),s>0)$ the excursions away from $x$ that occur before $\tau_0^-$. Then,
\begin{align*}
\widehat{\mathbb{E}}_x \left[e^{-\lambda L^x_{\tau_0^-} -\int_x^\infty f(y) L_{\tau_0^-}^y} \right] &=\int_0^\infty dt q_x e^{-tq_x -\lambda t} \widehat{\mathbb{E}}_x \left(\exp \left\{-\sum_{0<s<t} \int_x^\infty f(y) L_\zeta^y (\mathbf{e}_s^x (\tau_0^->\zeta_s )) dy \right\} \right) \\
&= \int_0^\infty dt q_x e^{-tq_x -\lambda t} e^{-t \widehat{N}_x\left(1-e^{-\int_x^\infty f(y) L_\zeta^y dy},\tau_0^- >\zeta \right)},\\
&= \int_0^\infty dt q_x e^{-t\left(q_x +\lambda + \widehat{N}_x\left(1-e^{-\int_x^\infty f(y) L_\zeta^y dy},\tau_0^- >\zeta \right)\right)}
\end{align*}
where the second equation holds by the exponential formula of Poisson point processes. Recalling that $u_{x}(f)=\widehat{N}_x\left(1-e^{- \int_x^\infty f(y) L_\zeta^y},\tau_0^- >\zeta\right)={N}_{0}\left(1-e^{-\int_x^\infty f(x-y) L_\zeta^y dy},\tau_{x}^{+} >\zeta\right),$ we conclude that $$\hat{\mathbb{E}}_x \left[1-e^{-\lambda L^x_{\tau_0^-} -\int_x^\infty f(y) L_\zeta^y} \right]=\hat{\mathbb{E}}_x \left[1-e^{-(\lambda+u_{x}(f)) L^x_{\tau_0^-}} \right].$$ Using this, we get
\begin{align*}
\overline{N}\left(1-e^{-\lambda L_\zeta ^x- \int_x^\infty f(y) L_\zeta^y dy} \right)& =\overline{N} \left(\left(1-e^{-\lambda L_\zeta ^x- \int_x^\infty f(y) L_\zeta^y dy} \right)1_{\{H>x\}}\right)\\
&\ = \overline{N} \left(\left(1-e^{-\int_x^\infty f(y) L_{T_x}^y dy} \right)1_{\{H>x\}}\right) \\
& \quad + \overline{N} \left(e^{-\int_x^\infty f(y) L_{T_x}^y dy} \hat{\mathbb{E}}_x \left[1-e^{-\lambda L^x_{\tau_0^-} - \int_x^\infty f(y) L_{\tau_0^-}^y} \right] 1_{\{H>x\}}\right) \\
& = \overline{N} \left(\left(1-e^{-\int_x^\infty f(y) L_{T_x}^y dy} \right)1_{\{H>x\}}\right) \\
& \quad + \overline{N} \left(e^{-\int_x^\infty f(y) L_{T_x}^y dy} \hat{\mathbb{E}}_x \left[1-e^{-(\lambda+u_{x}(f)) L^x_{\tau_0^-}} \right] 1_{\{H>x\}}\right) \\
&  = \overline{N} \left(\left(1-e^{-(\lambda+u_{x}(f)) L_{\zeta}^x \circ \theta_{T_x}-\int_x^\infty f(y) L_{T_x}^y dy} \right)1_{\{H>x\}}\right) \\
& \quad =\overline{N} \left(1-e^{-(\lambda+u_{x}(f)) L_{\zeta}^x -\int_x^\infty f(y) L_{T_x}^y dy} \right), \ \ \ \forall \lambda,\beta \geq 0,
\end{align*}
where the last expression follows from the fact that since an excursion does not accumulate local time at $x$ prior to its first visit to it, then $L_\zeta^x \circ \theta_{T_x}=L_\zeta^x$. The result on the joint law of local times at $n$ different points, $L_\zeta^{y_1},\dots,L_\zeta^{y_n}$ is proved in the same way, just replacing everywhere the integral by the sum $\sum_{k=1}^n \beta_k L_\zeta^{y_k}$. 
\end{proof}

\begin{proof}[Proof of Proposition \ref{ovN2}]
To prove this proceed as follows. By the Markov property under $\overline{N}$ applied at time $T_y$
\begin{equation*}
\begin{split}
u_y(\lambda)&=\overline{N}\left(1-\exp\{-\lambda L^{y}_{\zeta}\}\right)\\
&= \overline{N}\left(\left(1-\exp\{-\lambda L^{y}_{\zeta}\}\right)1_{\{H>y\}}\right)\\
&= \overline{N}\left(H>y\right)-\overline{N}\left(\exp\{-\lambda L^{y}_{\zeta}\}1_{\{H>y\}}\right)\\
&=\overline{N}\left(H>y\right)-\overline{N}\left(H>y\right)\widehat{\mathbb{E}}^{0}_y\left(\exp\{-\lambda L^{y}_{\zeta}\}\right)\\
&=\overline{N}\left(H>y\right)-\overline{N}\left(H>y\right)\widehat{\mathbb{E}}_{0}\left(\exp\{-\lambda L^{0}_{\tau^{-}_{-y}}\}\right)
\end{split}
\end{equation*}
Now, because $L^{0}_{\tau^{-}_{-y}}$ under $\widehat{\mathbb{E}}_{0}$ follows an exponential distribution with parameter $\widehat{N}_0(\tau^{-}_{-y}<\zeta).$ We have also calculated in the proof of Lemma~\ref{ovNx} that 
\begin{equation*}
\begin{split}
\widehat{N}_0(\tau^{-}_{-y}<\zeta)&={N}_0(\tau^{+}_{y}<\zeta)\\
&=N_{0}(\tau^{+}_y<\tau^{-}_0)=\frac{c_{+}}{W(y)},\qquad y>0.
\end{split}
\end{equation*}
The rest of the calculation now follows from the fact that 
$$\overline{N}\left(H>y\right)=\frac{W^{\prime}(y)}{W(y)},\quad y>0,$$ (see \cite[Lemma 8.2]{kyprianou2014fluctuations}).
\end{proof}

\begin{proof}[Proof of Proposition \ref{ovN3}]
Observe that $v_{x,y}(\lambda)=\widehat{N}_x^0\left(1-e^{-\lambda L_\zeta^y} \right)$. Since the excursion does not accumulate local time if the path does not reach $y$, then $$\widehat{N}_x^0\left(1-e^{-\lambda L_\zeta^y} \right)=\widehat{N}_x^0\left[\left(1-e^{-\lambda L_\zeta^y} \right)1_{\{T_y<\zeta \}} \right].$$ From the Markov property at time $T_y$,
\begin{align*}
v_{x,y}&=\widehat{N}_x^0\left[\left(1-e^{-\lambda L_\zeta^y} \right)1_{\{T_y<\zeta \}} \right]=\widehat{N}_x^0\left[1_{\{T_y<\zeta \}} \widehat{\mathbb{E}}_y \left(1-e^{-\lambda L_{T_x}^y} \right) \right] \\
&=\widehat{N}_x^0\left[T_y<\zeta \right] - \widehat{N}_x^0\left[T_y<\zeta \right] \widehat{\mathbb{E}}_0 \left(e^{-\lambda L_{T_{x-y}}^0} \right).
\end{align*}
Because of regularity and the absence of negative jumps, $T_{x-y}=\tau_{x-y}^-$ under $\widehat{\mathbb{E}}_0$. As in the previous proposition, $L_{\tau^-_{x-y}}^0$ follows an exponential distribution of parameter $\widehat{N}_0(\tau_{x-y}^-<\zeta)=N_0(\tau_{y-x}^+<\zeta)=\frac{c_+}{W(y-x)}$, which implies $\widehat{\mathbb{E}}_0 \left(e^{-\lambda L_{T_{x-y}}^0} \right)=\frac{c_+}{c_+ +\lambda W(y-x)}$. Therefore, $$v_{x,y}(\lambda)=\widehat{N}_x^0\left[T_y<\zeta \right] \frac{\lambda W(y-x)}{c_++\lambda W(y-x)}.$$ Now, under $\widehat{N}_x^0$ the event $\{T_y<\zeta,\tau_y^+>\zeta \}$ has zero measure because of regularity of 0 for $(0,\infty)$. Since there are no negative jumps under $\widehat{N}_x^0$, we also have $\{\tau_y^+<\zeta \} \subset \{ T_y<\zeta\}$. Hence, 
\begin{align*}
\widehat{N}_x^0\left[T_y<\zeta \right]&=\widehat{N}_x^0\left[T_y<\zeta,\tau_y^+>\zeta \right]+\widehat{N}_x^0\left[T_y<\zeta,\tau_y^+<\zeta \right]=\widehat{N}_x^0\left[\tau_y^+<\zeta \right]= \widehat{N}_x\left[\tau_y^+<\zeta<\tau_0^- \right]\\
&= \widehat{N}_0\left[\tau_{y-x}^+<\zeta<\tau_{-x}^- \right] \\
&= N_0 \left[\tau_{x-y}^- <\zeta <\tau_x^+ \right].
\end{align*}

To provide an expression of this in terms of $W$, we intersect the event inside $N_0$ with the partition $\mathcal{E}_+$, $\mathcal{E}_-$, $\mathcal{E}_\pm$. Excursions on $\mathcal{E}_+$ are completely above zero and therefore do not reach the negative level $x-y$. On $\mathcal{E}_-$ excursions are completely negative, and hence the condition $\tau_x^+>\zeta$ is fulfilled. Additionally, from Theorem 3 in \cite{pardo2018excursion} we know that $N_0$ on $\mathcal{E}_-$ is a multiple of $\widehat{\overline{N}}$, which is the pushforward of $\overline{N}$ under the map that sends each path to its negative. Therefore,
\begin{align*}
N_0 \left[\tau_{x-y}^- <\zeta <\tau_x^+,\mathcal{E}_- \right]&= N_0 \left[\tau_{x-y}^- <\zeta,\mathcal{E}_- \right] = \frac{c_-\sigma^2}{2} \widehat{\overline{N}} \left[\tau_{x-y}^- <\zeta \right]= \frac{c_-\sigma^2}{2} \overline{N} \left[\tau_{y-x}^+ <\zeta \right] \\
&= \frac{c_-\sigma^2}{2} \overline{N} \left[H>y-x \right] \\
&= \frac{c_-\sigma^2}{2} \frac{W'(y-x)}{W(y-x)},
\end{align*}
where the last line follows from the fact that $W$ is differentiable and $W'(z)=W(z)\overline{N}(H>z)$ (see \cite[Lemma 8.2]{kyprianou2014fluctuations}).

Finally, on $\mathcal{E}_\pm$ excursions start above zero, then jump below and die at the next hitting time at zero. Therefore,
\begin{align*}
N_0 \left[\tau_{x-y}^- <\zeta <\tau_x^+,\mathcal{E}_\pm \right]&= N_0 \left[\tau_{x-y}^- <\zeta <\tau_x^+,0<\tau_0^-<\zeta \right] \\
&= N_0 \left[\sup_{s\in (0,\tau_0^-]} \mathbf{e}(s) <x, \tau_{x-y}^- <\zeta,0<\tau_0^-<\zeta \right].
\end{align*}

From the Markov property at time $\tau_0^-$, the excursion after the jump follows the same law as $X$ started from $\mathbf{e}(\tau_0^-)$ and killed at the first passage time above zero. This fact together with a conditioning on the time the jump occurs implies that the expression on the right hand side of the display above equals
\begin{align*}
& N_0 \left[1_{\left\{\sup_{s\in (0,\tau_0^-]} \mathbf{e}(s) <x \right\}} \mathbb{P}_{\mathbf{e}(\tau_0^-)} \left(\tau_{x-y}^- < \tau_0^+  \right) \right] \\
&= N_0 \left[1_{\left\{\sup_{s\in (0,\tau_0^-]} \mathbf{e}(s) <x \right\}} 1_{\{\mathbf{e}(\tau_0^--) >0, \mathbf{e}(\tau_0^-) <0\}} \mathbb{P}_{\mathbf{e}(\tau_0^-)} \left(\tau_{x-y}^- < \tau_0^+  \right) \right] \\
&= N_0 \left[ \int_0^\infty dt 1_{\left\{\sup_{s\in (0,t]} \mathbf{e}(s) <x \right\}} 1_{\{\mathbf{e}(t-) >0 \}} \int_{-\infty}^{0} \Pi(du) 1_{\{\mathbf{e}(t-)+u<0 \}}\mathbb{P}_{\mathbf{e}(t-)+u} \left(\tau_{x-y}^- < \tau_0^+  \right) \right]\\
&= N_0 \left[ \int_0^\infty dt 1_{\left\{\sup_{s\in (0,t]} \mathbf{e}(s) <x \right\}} 1_{\{\mathbf{e}(t) >0 \}} \int_{-\infty}^{0} \Pi(du) 1_{\{\mathbf{e}(t)+u<0 \}} \mathbb{P}_{\mathbf{e}(t)+u} \left(\tau_{x-y}^- < \tau_0^+  \right) \right] \\
&= \int_0^\infty dt N_0 \left[  1_{\left\{\sup_{s\in (0,t]} \mathbf{e}(s) <x \right\}} 1_{\{\mathbf{e}(t) >0 \}} h(\mathbf{e}(t)) \right],
\end{align*}
where $$h(z)=\int_{-\infty}^{0} \Pi(du) 1_{\{z+u<0 \}} \mathbb{P}_{z+u} \left(\tau_{x-y}^- < \tau_0^+  \right)=\int_{-\infty}^{-z} \Pi(du) \mathbb{P}_{z+u} \left(\tau_{x-y}^- < \tau_0^+  \right), \quad z>0.$$

Observe that if $u\leq -z-(y-x)$ then $\mathbb{P}_{z+u} \left(\tau_{x-y}^- < \tau_0^+  \right)=1$. On the other hand, if $-z-(y-x)<u<-z$ then the latter probability is just the exit problem from the interval $[x-y,0]$ starting from $z+u$, so $\mathbb{P}_{z+u} \left(\tau_{x-y}^- < \tau_0^+  \right)=1-\frac{W(u+z+y-x)}{W(y-x)}$. 

From the same calculation as in the proof of Lemma~\ref{ovNx}, we know that there exists $c_+>0$ such that
\begin{align*}
&N_0 \left[  1_{\left\{\sup_{s\in (0,t]} \mathbf{e}(s) <x \right\}} 1_{\{\mathbf{e}(t) >0 \}} h(\mathbf{e}(t)) \right]\\
&=\int_0^x dz e^{-\Phi(0)z} h(z) - \frac{c_+}{W(x)}\int_0^x dz \left(e^{-\Phi(0)z}W(x)-W(x-z) \right)h(z) \\
&= \int_0^x dz \left( 1-c_+ + \frac{c_+ W(x-z)}{W(x)} \right) h(z) \\
&= \int_0^x dz \left( 1-c_+ + \frac{c_+ W(x-z)}{W(x)} \right) \left(\Pi(-\infty,z)-\int_{-z-(y-x)}^{-z} \Pi(du) \frac{W(u+z+y-x)}{W(y-x)} \right).
\end{align*}

Putting the previous computations together we conclude that
\begin{align*}
\widehat{N}_x^0\left[T_y<\zeta \right]&= \frac{c_-\sigma^2}{2} \frac{W'(y-x)}{W(y-x)} \\
& + \int_0^x dz \left( 1-c_+ + \frac{c_+ W(x-z)}{W(x)} \right) \left(\Pi(-\infty,-z)-\int_{-z-(y-x)}^{-z} \Pi(du) \frac{W(u+z+y-x)}{W(y-x)} \right).
\end{align*}
\end{proof}

\section*{Acknowledgements}
This work is part of the first author's research project in Centro de Investigaci\'on en Matem\'aticas. This author has been supported by the PhD grant (CVU number 706815) from the National Council of Science and Technology (CONACYT). Both authors thank Wei Xu for the several insightful discussions about this work.

\bibliographystyle{abbrv}
\bibliography{rkbib}

\end{document}